\documentclass[draft]{amsproc}

\theoremstyle{plain} 

\newtheorem{theorem}{Theorem}[section]
\newtheorem{lemma}[theorem]{Lemma}
\newtheorem{corollary}[theorem]{Corollary}
\newtheorem{conjecture}[theorem]{Conjecture}

\newtheorem{proposition}[theorem]{Proposition}
\newtheorem{question}[theorem]{Question}

\newcommand{\eop}{\ \hfill $\Box$}

\numberwithin{equation}{section}

\newcommand{\cc}{{\mathbb C}}
\newcommand{\pp}{{\mathbb P}}
\newcommand{\rr}{{\mathbb R}}
\newcommand{\qq}{{\mathbb Q}}
\newcommand{\zz}{{\mathbb Z}}

\newcommand{\ggg}{{\mathbb G}}
\newcommand{\aaa}{{\mathbb A}}
\newcommand{\Gm}{\ggg _m}

\newcommand{\Ee}{{\mathcal E}}
\newcommand{\Pp}{{\mathcal P}}
\newcommand{\Cc}{{\mathcal C}}

\newcommand{\Gg}{{\mathcal G}}
\newcommand{\Dd}{{\mathcal D}}
\newcommand{\Ff}{{\mathcal F}}
\newcommand{\Oo}{{\mathcal O}}

\newcommand{\delbar}{\overline{\partial}}

\begin{document}

\author[C. Simpson]{Carlos Simpson}
\address{CNRS, Laboratoire J. A. Dieudonn\'e, UMR 6621
\\ Universit\'e de Nice-Sophia Antipolis\\
06108 Nice, Cedex 2, France}
\email{carlos@math.unice.fr}
\urladdr{http://math.unice.fr/$\sim$carlos/} 

\title[Weight two phenomenon]{A weight two phenomenon for the moduli of rank one
local systems on open varieties}

\subjclass{Primary 14D21, 32J25; Secondary 14C30, 14F35}

\keywords{Connection, Fundamental group, Higgs bundle, Parabolic structure, Quasiprojective variety, Representation, Twistor space}

\begin{abstract}
The twistor space of representations on an open variety maps to a weight two space of local monodromy
transformations around a divisor component at infinty. The space of $\sigma$-invariant sections
of this slope-two bundle over the twistor line is a real $3$ dimensional space whose parameters correspond to the complex residue of the
Higgs field, and the real parabolic weight of a harmonic bundle. 
\end{abstract}

\maketitle


\section{Introduction} \label{sec-introduction}

Let $X$ be a smooth projective variety and $D\subset X$ a reduced effective divisor with simple normal crossings. 
We would like to define a Deligne glueing for the Hitchin twistor space 
of the moduli of local systems over $X-D$. Making the construction presents
new difficulties which are not present in the case of compact base, so we only treat the case of local systems
of rank $1$. Every local system comes from
a vector bundle on $X$ with connection logarithmic along $D$, however one can make local meromorphic
gauge transformations near components of $D$, and this changes the structure of the bundle as well as the
eigenvalues of the residue of the connection. The change in eigenvalues is by subtracting an integer. 
There is no reasonable algebraic quotient by such an operation: for our main example \S \ref{sec-p1}, that
would amount to taking the quotient of $\aaa ^1$ by the translation action of $\zz$. Hence, we are tempted
to look at the moduli space of logarithmic connections and accept the fact that the Riemann-Hilbert
correspondence from there to the moduli space of local systems, is many-to-one. 

We first concentrate on looking at the simplest case, which is when 
$X:= \pp ^1$ and $D:= \{ 0 , \infty \} $ and the local systems have rank $1$. In this case,
much as in Goldman and Xia \cite{GoldmanXia}, one can explicitly write down everything, in particular
we can write down a model. This will allow observation of the weight two phenomenon which is new in the 
noncompact case. 

The residue of a connection takes values in a space which represents the
local monodromy around a puncture. As might be expected, this space has weight two, so when we do the Deligne glueing
we get a bundle of the form $\Oo _{\pp ^1}(2)$. There is an antipodal involution $\sigma$ on this bundle, and the preferred sections
corresponding to harmonic bundles are $\sigma$-invariant.  The
space of $\sigma$-invariant sections of $\Oo _{\pp ^1}(2)$ is $\rr ^3$, in particular it doesn't map isomorphically to a fiber
over one point of $\pp ^1$. Then kernel of the map to the fiber is the parabolic weight parameter. 
Remarkably, the parabolic structure appears ``out of nowhere'', as a result of the holomorphic structure
of the Deligne-Hitchin twistor space constructed only using the notion of logarithmic $\lambda$-connections. 

After \S \ref{sec-p1} treating in detail the case of $\pp ^1 - \{ 0 , \infty \} $, 
we look in \S \ref{sec-tatetwistor} more closely at the bundle $\Oo _{\pp ^1}(2)$ which occurs: it is the {\em Tate twistor structure},
and is also seen as a twist of the tangent bundle  $T\pp ^1$. 
Then \S \ref{sec-general} concerns the case of rank one local systems when $X$ has arbitrary dimension.  In 
\S \ref{sec-strict} we state a conjecture about strictness which should follow from a full mixed theory as we are suggesting here. 

Since we are considering rank one local systems,
the tangent space is Deligne's mixed Hodge structure on $H^1(X -D, \cc )$ (see Theorem \ref{mhs-ident}). 
However, a number of authors, such as Pridham \cite{PridhamMHS} \cite{PridhamQl} and
Brylinski-Foth \cite{BrylinskiFoth} \cite{Foth} have already constructed and studied a mixed Hodge structure on the deformation space of
representations of rank $r>1$ over an open variety. These structures should amount to the local version of what we
are looking for in the higher rank case, and motivate the present paper. They might also allow a direct proof of the
infinitesimal version of the strictness conjecture \ref{conj-strictness}. 

In the higher rank case, there are a number of problems blocking a direct generalization of what we do here. 
These are mostly related to non-regular monodromy operators. In a certain sense, the local structure of a connection with
diagonalizable monodromy operators, is like the direct sum of rank $1$ pieces. However, the action of the gauge group contracts to a trivial
action at $\lambda = 0$, so there is no easy way to cut out an open substack corresponding only to regular values. 
We leave this generalization as a problem for future study. This will necessitate using contributions from other works in the subject, such as
Inaba-Iwasaki-Saito \cite{InabaIwasakiSaito} \cite{InabaIwasakiSaito2} and Gukov-Witten
\cite{GukovWitten}.

This paper corresponds to my talk in the conference ``Interactions with Algebraic Geometry'' in Florence (May 30th-June 2nd 2007), 
just a week after the Augsberg conference. Sections \ref{sec-tatetwistor}--\ref{sec-strict} were added later. We hope that the observation we
make here can contribute to some understanding of this subject, which is related to a number of other works such as the notion of 
$tt^{\ast}$-geometry \cite{Hertling} \cite{Schafer}, geometric Langlands theory \cite{GukovWitten}, Deligne cohomology \cite{EsnaultViehweg} \cite{Gajer}, 
harmonic bundles \cite{Biquard} \cite{Mochizuki} and twistor $\Dd$-modules \cite{Sabbah}, 
Painlev\'e equations \cite{Boalch} \cite{InabaIwasakiSaito} \cite{InabaIwasakiSaito2}, and
the theory of rank one local systems on open varieties
\cite{Budur} \cite{Dimca} \cite{DimcaMaisonobeSaito}
\cite{DimcaPapadimaSuciu} \cite{Libgober}. 

\section{Preliminary definitions}
\label{sec-prelim}

It is useful to follow Deligne's way of not choosing a square root of $-1$. This serves as a guide to making constructions
more canonically, which in turn serves to avoid encountering unnecessary choices later. We do this because one of the goals below is
to understand in a natural way the Tate twistor structure $T(1)$. In particular, this has served as a useful guide for finding
the explanation given in \S \ref{sub-tate-integer} for the sign change necessary in the logarithmic version $T(1,\log )$. 
We have tried, when possible, to explain the motivation for various other minus signs too. 
{\em Caution:} there may remain sign errors specially towards the end. 

Let $\cc$ be an algebraic closure of $\rr$, but without a chosen $\sqrt{-1}$. 
Nevertheless, occasional explanations using a choice of $i = \sqrt{1} \in \cc$ are admitted
so as not to leave things too abstruse.

\subsection{Complex manifolds}
\label{sub-complex}
There is a notion of $\cc$-linear complex manifold $M$. This
means that at each point $m\in M$ there should be an action of $\cc$ on the real tangent space $T_{\rr}(M)$. Holomorphic functions are
functions $M\rightarrow \cc$ whose $1$-jets are compatible with this action. Usual Hodge theory still goes through without refering to
a choice of $i\in \cc$. We get the spaces $A^{p,q}(M)$ of forms on $M$, and the operators $\partial$ and $\overline{\partial}$. 

Let $\rr ^{\perp}$ denote the imaginary line in $\cc$. This is what Deligne would call $\rr (1)$ however we don't divide by $2\pi$. 

If $h$ is a metric on $M$, there is a naturally associated two-form $\omega\in A^2(M, \rr ^{\perp})$.  
The K\"ahler class is $[\omega ]\in H^2(X,\rr ^{\perp}) = H^2(X,\rr (1))$. 
Classically this is brought back to
a real-valued $2$-form by multiplying by a choice of $\sqrt{-1}$, but we shouldn't do that here. Then, the operators $L$ and $\Lambda$ are
defined independently of $\sqrt{-1}$, but they take values in $\rr ^{\perp}$. The K\"ahler identities now hold without $\sqrt{-1}$ appearing; 
but it is left to the reader to establish a convention for the signs.  

Note that $M$ may not be canonically oriented. If $Q = \{ \pm \sqrt{-1}\}$ as below, then the orientation of $M$ is canonically defined
in the $n$-th power $Q^n\subset \cc$ where $n=dim_{\cc}M$. In particular, the orientation in codimension $1$ is always ill-defined. 
If $D$ is a divisor, this means that $[D]\in H^2(M, \rr ^{\perp})$. This agrees with what happens with the K\"ahler metric. 
Similarly, if $L$ is a line bundle then $c_1(L)\in H^2(M, \rr ^{\perp})$. 

If $X$ is a quasiprojective variety over $\cc$ then $X(\cc )$ has a natural topology. Denote this topological space by $X^{\rm top}$.
It is the topological space underlying a structure of complex analytic space. In the present paper, we don't distinguish too much between algebraic
and analytic varieties, so we use the same letter $X$ to denote the analytic space. 

Let $\overline{X}$ denote the conjugate variety, where the structural map is composed with the complex conjugation $Spec(\cc )\rightarrow Spec(\cc )$.
In terms of coordinates, $\overline{X}$ is given by equations whose coefficients are the complex conjugates of the coefficients of the equations of $X$.
There is a natural isomorphism 
$\varphi : X^{\rm top}\stackrel{\cong}{\rightarrow} \overline{X}^{\rm top}$, which in terms of equations is given by $x\mapsto \overline{x}$ conjugating
the coordinates of each point.

\subsection{The imaginary scheme of a group}
\label{sub-imaginary}
Let $Q\subset \cc$ be the zero set of the polynomial
$x^2+1$, in other words $Q = \{ \pm \sqrt{-1}\}$. Multiplication by $-1$ is equal to multiplicative inversion, which is equal to complex
conjugation, and these all define an involution 
$$
c_Q:Q\rightarrow Q.
$$

Suppose $Y$ is a set provided with an involution $\tau _Y$. Then we define a new set denoted $Y^{\perp}$ starting from $Hom (Q,Y)$ with its two involutions 
$$
f\mapsto \tau _Y\circ f, \;\;\; f\mapsto f\circ c_Q.
$$
Let $Y^{\perp}$ be the equalizer of these two involutions, in other words 
$$
Y^{\perp}:= \{ f\in Hom (Q,Y), \;\; \tau _Y\circ f = f\circ c_Q \} . 
$$
Thus, an element of $G^{\perp}$ is a function $\gamma : q \mapsto \gamma (q)$ such that $\gamma (-q)= \tau _Y(\gamma (q))$.

The two equal involutions will be denoted $\tau _{Y^{\perp}}$. 

If we choose $i = \sqrt{-1}\in \cc$, then $Y^{\perp}$ becomes identified with $Y$ via $\gamma \mapsto \gamma (i)$. 
For the opposite choice of $i$ this isomorphism gets composed with $\tau _Y$. 

If $G$ is a scheme defined over $\rr$ then $G^{\perp}$ is also defined over $\rr$. For example, if $G=\rr$ with involution $x\mapsto -x$
then $G^{\perp}$ is the imaginary
line $\rr ^{\perp}$ defined above. 

Using the involution $x\mapsto -x$ we could also define $\cc ^{\perp}$. However there is a natural isomorphism $\cc \cong \cc ^{\perp}$
sending $a$ to the function $\gamma : q \mapsto qa$.  In view of this, and in order to lighten notation, we don't distinguish between
$\cc$ and $\cc ^{\perp}$ even in places where that might be natural for example throughout \S \ref{sec-exact}.  

If $G$ is an algebraic group over $\cc$, it has an involution $g\mapsto g^{-1}$, which doesn't preserve
the group structure unless $G$ is abelian. Using this involution yields a scheme denoted $G^{\perp}$.
If $G$ is abelian then $G^{\perp}$ has a natural group structure. In general
there is a natural action of $G$ on $G^{\perp}$ by conjugation: if $g\in G$ and 
$q\mapsto \gamma (q)$ is an element of $G^{\perp}$ then the element $q\mapsto g\gamma (q)g^{-1}$ is again an element of $G^{\perp}$. 

Since we will be looking mostly at rank one local systems, we are particularly interested in the case $G=\Gm$. Then
$$
\Gm ^{\perp} = \{ (x,y)\in \cc ^2, \;\; x^2 + y^2 = 1 \} .
$$
The equality is given as follows: to an element $\gamma : q \mapsto \gamma (q)$ of $\Gm ^{\perp}$, associate
the point $(x,y)$ given by
$$
x:= \frac{1}{2}\sum _{q\in Q}\gamma (q), 
$$
$$
y:= \frac{1}{2}\sum _{q\in Q}q^{-1}\gamma (q). 
$$
Call $(x,y)$ the {\em circular coordinates} on $\Gm ^{\perp}$. 

It is well-known that the exponential should really be considered as a map ${\rm exp}: \cc (1) \rightarrow \Gm$. 
Alternatively, we can view the exponential as a map 
$$
{\rm exp}^{\perp} : \cc \rightarrow \Gm ^{\perp}.
$$
given in circular coordinates by 
$$
{\rm exp}^{\perp}(\theta ) := (\cos (2\pi \theta ), \sin (2\pi \theta )).
$$
It is useful to include $2\pi$ here because of the relationship with residues, see below.
The kernel of ${\rm exp}^{\perp}$ is the usual $\zz \subset \cc$. We call $\theta$ a {\em circular logarithm}
of its image point.

\subsection{Logarithmic connections}
\label{sub-logarithmic}
Suppose $X$ is a smooth projective variety and $D\subset X$ is a normal crossings divisor.
Let $U:= X-D$ and $j: U\hookrightarrow X$ be the inclusion. 
Recall that the sheaf of {\em logarithmic forms} on $(X,D)$ denoted $\Omega ^1_X(\log D)$ is the
locally free sheaf, subsheaf of $j_{\ast}\Omega ^1_{U}$,
which is generated in local coordinates by $d\log z_1,\ldots , d\log z_k , dz_{k+1},\ldots , dz_n$
whenever $(z_1,\ldots , z_n)$ is a system of local coordinates in which $D$ is given by $z_1\cdots z_k = 0$.

A {\em logarithmic connection} $\nabla$ on a vector bundle $E$ over $X$, is a morphism of sheaves
$$
\nabla : E\rightarrow E\otimes _{\Oo _X}\Omega ^1_X(\log D)
$$
such that $\nabla (af) = a\nabla (f) + da \cdot f$. More generally, for $\lambda \in \cc$
a {\em logarithmic $\lambda$-connection}
is a map $\nabla$ as above such that $\nabla (af) = a\nabla (f) + \lambda da \cdot f$. For $\lambda = 1$ this is a usual connection,
and for any $\lambda \neq 0$ we get a usual connection $\lambda ^{-1}\nabla$. For $\lambda = 0$ it is a Higgs field. 

The {\em Riemann-Hilbert correspondence} takes a vector bundle with logarithmic connection $(E,\nabla )$ to its monodromy
representation $\rho$. 
This is well-defined independent of the choice of $\sqrt{-1}\in \cc$. 
In the compact case, it is an equivalence of categories between vector bundles with connection, and 
representations up to conjugacy. However, in our open case there are many possible
choices of $(E,\nabla )$ which give the same representation $\rho$, because of the possibility of making {\em meromorphic gauge transformations}
along the components of the divisor $D$, see \S \ref{sub-meromorphicgauge} below. 
For any $\lambda \neq 0$, the {\em monodromy representation} of a $\lambda$-connection is by definition that of
the normalized connection $\lambda ^{-1}\nabla$. 

\subsection{Local monodromy}
\label{sub-localmonodromy}

The reason for introducing the imaginary scheme $G^{\perp}$ was to discuss local monodromy. Keep the notation that $(X,D)$ is a smooth variety
with a normal crossings divisor. For each component $D_i$ of $D$, choose a point $x_i$ near $D_i$. 
Choose a local coordinate system $(z_1,\ldots , z_n)$ for $X$ near a smooth point of $D_i$, such that $D_i$ is given by $z_1=0$ and 
$x_i$ is the point $(\epsilon , 0 , \ldots , 0)$. 
We get a map from $Q$ to $\pi _1(X,x_i)$ as
follows: for $q\in Q$, consider the path $t\mapsto (\epsilon \cdot e^{2\pi q t}, 0, \ldots , 0)$. For $-q$ we get the inverse path,
in other words we really have an element of $\pi _1(X,x_i) ^{\perp}$. Conjugating by a choice of path from $x$ to $x_i$, we get an
element 
$$
\gamma _{D_i} \in \pi _1(X,x)^{\perp}.
$$ 
It is well-defined up to the conjugation action of $\pi _1(X,x)$.

If $\rho : \pi _1(X,x)\rightarrow G$ is a representation, we obtain by functoriality of the construction $(\; ) ^{\perp}$
a map 
$$
\rho ^{\perp}: \pi _1(X,x)^{\perp}\rightarrow G^{\perp},
$$
so we get the {\em local monodromy element} 
$$
{\rm mon}(\rho , D_i):= \rho ^{\perp} (\gamma _{D_i}) \in G^{\perp},
$$ 
which is well-defined up to the conjugation action of $G$. If $G$ is abelian, such as $G=\Gm$, then the local monodromy element is well-defined.

\subsection{Meromorphic gauge group}
\label{sub-meromorphicgauge}

Since we will mostly be working with line bundles, we describe the meromorphic gauge group only in this case. It is much easier than in
general. Decompose $D=D_1+\ldots + D_k$ into a union of smooth irreducible components. The gauge group is just 
$$
\Gg := \zz ^k ,
$$
acting as follows. Suppose $(L,\nabla )$ is a line bundle with logarithmic $\lambda$-connection on $(X,D)$ and 
$g=(g_1,\ldots , g_k)$ is an element of $\Gg$. Then the new line bundle  is
defined by
$$
L^g:= L(g_1D_1 + \ldots + g_kD_k),
$$
and $\nabla ^g$ is the unique logarithmic $\lambda$-connection on $L^g$ which coincides with $\nabla$ over the open set $U$
via the canonical isomorphism $L^g|_{U}\cong L|_{U}$. 

The gauge transformation affects the first Chern class:
\begin{equation}
\label{c1formula}
c_1(L^g) = c_1(L) + g_1[D_1] + \ldots + g_k [D_k] ,
\end{equation}
and the residue:
\begin{equation}
\label{residueformula}
{\rm res}(\nabla ^g ; D_i ) = {\rm res}(\nabla ; D_i) - \lambda g_i.
\end{equation}

For convenience, here is the proof of \eqref{residueformula}. 
If $u$ is a nonvanishing holomorphic ection of $L$ near a point of $D_i$ (but not near the other divisor components), then 
$u$ may also be considered as a meromorphic section of $L^g$, but it has a zero of order $g_i$ along $D_i$.
Hence, $u':=z_i^{-g_i}u$ is a nonvanishing holomorphic section of $L^g$ near our point of $D_i$. 

Let $R_i:= {\rm res}(\nabla ; D_i)$, so 
$$
\nabla (au) = \lambda d(a)u + R_i \frac{dz_i}{z_i}a u + \ldots .
$$
Generically, $\nabla $ and $\nabla ^g$ are the same connection. However, a section of $L^g$ is written in terms of the unit section $u'$ as $au' = az_i^{-g_i}u$,
so
$$
\nabla ^g (au') = \nabla (az_i^{-g_i}u) = \lambda d(a) u' - \lambda g_i \frac{dz_i}{z_i}a u' + R_i \frac{dz_i}{z_i}a u +\ldots .
$$
The residue of $\nabla ^g$ is
${\rm res}(\nabla ^g ; D_i ) = R_i - \lambda g_i$
as claimed in \eqref{residueformula}.

The restrictions to the open set are isomorphic:
$$
(L^g, \nabla ^g) |_{U} \cong (L,\nabla )|_{U},
$$
hence the monodromy representations are the same in the case $\lambda \neq 0$. Conversely, again in the case $\lambda \neq 0$,
given $(L,\nabla )$ and $(L',\nabla ')$ two logarithmic $\lambda$-connections with the same monodromy representations,
there is a unique meromorphic gauge transformation $g\in\Gg$ such that $(L',\nabla ') \cong (L^g, \nabla ^g)$. 

Throughout the paper, make the convention that spaces and maps are in the complex analytic category. The reader will notice
which parts of these analytic spaces have natural algebraic structures, for example the Betti spaces or the charts $M_{\rm Hod}(X,\log D)$. 
Often these algebraic charts will be divided by a group action or glued to other charts in an analytic way, so the result only has a
structure of analytic space.

\section{The Deligne glueing in the compact case}
\label{sec-deligneglue}

In this section we recall the Deligne glueing construction for the twistor space,
in the case of a compact base variety $X$, that is $D=\emptyset$. The hyperk\"ahler structure
on the moduli space was constructed by Hitchin \cite{Hitchin}, who also considered the Penrose twistor space
associated to the quaternionic structure. Deligne in \cite{DeligneLett} proposed a construction of the
twistor space using a deformation called the space of {\em $\lambda$-connections} closely related to the Hodge filtration, 
plus the Riemann-Hilbert correspondence relating connections on $X$ and the conjugate variety $\overline{X}$.
Apparently Witten contributed 
something too because Deligne's letter \cite{DeligneLett} starts off: 
\begin{quotation}
{\small
``As I understand, Hitchin's understanding of why one has
a hyperk\"ahler structure---as explained to me by Witten---works in your case. \ldots''.}
\end{quotation}
The twistor space structure is  related to the notion of $tt^{\ast}$ geometry \cite{CecottiVafa} \cite{Hertling} \cite{Schafer}. 
The idea of a deformation relating de Rham and
Dolbeault cohomology goes back further, to the theory of $\Gamma$-factors \cite{Deninger}, 
Esnault's notion of $\tau$-connection \cite{Esnault}, Dolbeault homotopy theory
\cite{NeisendorferTaylor}, to the relation between cyclic and Hochschild cohomology
\cite{Connes} \cite{Kaledin}, and to singular perturbation theory \cite{Voros}. 

\subsection{Moduli spaces}
\label{sec-moduli}
Fix a basepoint $x\in X$. Complex conjugation provides a map of topological spaces $\varphi : X^{\rm top}\rightarrow 
\overline{X}^{\rm top}$, which is antiholomorphic for the complex structures. In particular, we get
$$
\varphi _{\ast}: \pi _1(X,x)\stackrel{\cong}{\rightarrow} \pi _1(\overline{X},\overline{x}).
$$

Recall the following moduli spaces or moduli stacks. Usually we don't distinguish
between moduli stacks or their universal categorical quotients which are moduli spaces. Also we are fixing
the target group as $GL(r,\cc )$ which will be left out of the notation. In this section we let $r$ be arbitrary,
although the next sections will specialize to $r=1$. 

Write $M_{\rm Hod}(X)\rightarrow \aaa ^1$ for the moduli space or stack of semistable vector bundles of rank $r$ with $\lambda$-connection
with vanishing Chern classes. The fibers over $0$ and $1$ are denoted respectively
$M_{\rm Dol}(X)$ and $M_{\rm DR}(X)$.
The group $\Gm$ acts, and over $\Gm \subset \aaa ^1$ this action provides
an isomorphism 
$$
M_{\rm Hod}(X)\times _{\aaa ^1}\Gm \cong \Gm \times M_{\rm DR}(X).
$$

These natural constructions applied to the conjugate variety give conjugate varieties:
$$
M_{\rm Hod}(\overline{X}) \cong \overline{M_{\rm Hod}(X)},
\;\;\;
M_{\rm DR}(\overline{X}) \cong \overline{M_{\rm DR}(X)},
\;\;\;
M_{\rm Dol}(\overline{X}) \cong \overline{M_{\rm Dol}(X)}.
$$

We also have the Betti moduli space \cite{LubotskyMagid}
$$
M_{\rm B}(X) = \frac{{\rm Hom}(\pi _1(X,x), GL(n,\cc ))}{GL(n,\cc )}
$$
where the quotient is either a stack quotient or a universal categorical quotient depending on 
which framework we are using. The
Riemann-Hilbert correspondence gives an isomorphism of analytic spaces or stacks
$$
M_{DR}(X)^{\rm an} \cong M_{B}(X)^{\rm an}.
$$
It doesn't depend on a choice of square root of $-1$. 

\subsection{Glueing}
\label{sec-glue}
The {\em Deligne glueing} is an isomorphism of complex analytic spaces
$$
{\bf d} : M_{\rm Hod}(X)\times _{\aaa ^1}\Gm  \cong M_{\rm Hod}(\overline{X})\times _{\aaa ^1}\Gm  .
$$
It is defined as follows. A point in the source is a triple $(\lambda , E, \nabla )$ where
$\lambda \in \Gm \subset \aaa ^1$, where $E$ is a vector bundle on $X$, and $\nabla $ is
a $\lambda$-connection on $E$. This corresponds to the point $(\lambda , (E, \lambda ^{-1}\nabla ))$
in $\Gm \times M_{\rm DR}(X)$. Let $\rho (\lambda ^{-1}\nabla )$ denote the monodromy
representation of $\pi _1(X,x)$ corresponding to 
the connection $\lambda ^{-1}\nabla$. Then $\rho (\lambda ^{-1}\nabla ) \circ \varphi _{\ast}^{-1}$ 
is a representation of $\pi _1(\overline{X},\overline{x})$. It corresponds to a
vector bundle with connection $(F, \Phi )$ on $\overline{X}$. 

This vector bundle with connection may be characterized as follows:
\newline
---we have a natural identification $F_{\overline{x}}\cong E_x$; and 
\newline
---the monodromy of $(F,\Phi )$ around a loop $\gamma $ in $\pi _1(\overline{X}, \overline{x})$
is equal, via this identification, to the monodromy of $(E,\lambda ^{-1}\nabla )$ around 
the loop $\varphi ^{-1}(\gamma )$ in $\pi _1(X,x)$.

To continue with the definition of ${\bf d}$, 
choose the
point $\mu = \lambda ^{-1} \in \Gm$, and look at the point 
$$
(\mu , (F,\Phi ))\in \Gm \times M_{\rm DR}(\overline{X}).
$$
It corresponds to a point 
$$
(\mu , F, \mu \Phi )\in M_{\rm Hod}(\overline{X})\times _{\aaa ^1}\Gm .
$$
We set 
$$
{\bf d}(\lambda , E, \nabla ):= (\mu , F, \mu \Phi ).
$$
Note that by definition ${\bf d}$ covers the map $\Gm \rightarrow \Gm $ given by $\lambda \mapsto \mu := \lambda ^{-1}$.

This isomorphism can now be used to glue together the two analytic spaces 
$M_{\rm Hod}(X)^{\rm an}$ and $M_{\rm Hod}(\overline{X})^{\rm an}$ along their open sets
which are the source and target of ${\bf d}$. The resulting space is denoted $M_{\rm DH}(X)$
for {\em Deligne-Hitchin}. It is Hitchin's twistor space \cite{Hitchin}, constructed as suggested by Deligne
\cite{DeligneLett}. 

Interpreting $\pp ^1$ as obtained by glueing two copies of $\aaa ^1$ along
the map $\mu = \lambda ^{-1}$, we get a map $M_{\rm DH}(X)\rightarrow \pp ^1$. 

It is essential to make some remarks on the choices above. The space $M_{\rm Hod}(X)$ and its 
conjugate counterpart admit numerous natural automorphisms, for example multiplication by any 
element of $\Gm$, but also taking the dual of an object. In particular, it would have been possible
to insert these operations in the middle of the definition of ${\bf d}$. They would extend
to automorphisms of either of the two sides being glued, so the resulting space would be isomorphic.
We feel that it is reasonable
at each step of the way to use the simplest choice. 
This will nonetheless result in more complicated choices in the definition of preferred sections later. 

Note that we have not used any choice of $\sqrt{-1}\in \cc$ in the construction, so $M_{\rm DH}(X)$ is independant of that. 

According to the construction, notice that we have two inclusions
$$
u: M_{\rm Hod}(X)\hookrightarrow M_{\rm DH}(X), \;\;\; v:M_{\rm Hod}(\overline{X})\hookrightarrow M_{\rm DH}(X).
$$
We have $u(\lambda , E,\nabla ) = v (\mu , F , \mu \Phi )$ exactly when 
${\bf d}(\lambda , E, \nabla ) = (\mu , F, \mu \Phi )$ as constructed above.

We leave to the reader the problem of comparison of $M_{\rm DH}(X)$ and $M_{\rm DH}(\overline{X})$. 

\subsection{The antipodal involution}
\label{sec-antipodal}
A crucial part of the structure is an antilinear involution $\sigma : M_{\rm DH}(X)\rightarrow M_{\rm DH}(X)$, covering
the antipodal involution of $\pp ^1$. 
Note that the antipodal involution exchanges the two charts $\aaa ^1$ of $\pp ^1$. Thus, in order to
define $\sigma$ it suffices to define an antiholomorphic map
$$
\sigma _{{\rm Hod}, X} : M_{\rm Hod}(X)\rightarrow M_{\rm Hod}(\overline{X})
$$
which is an antilinear isomorphism, and involutive:
that is $\sigma _{{\rm Hod},\overline{X}}\circ \sigma _{{\rm Hod},X} = {\rm Id}$. 

Suppose we have a point $(\lambda , E, \nabla )$. Taking the complex conjugate of everything gives a point
$(\overline{\lambda}, \overline{E},\overline{\nabla})\in M_{\rm Hod}(\overline{X})$. This gives an
antiholomorphic map denoted
$$
C _{{\rm Hod}, X} : M_{\rm Hod}(X)\rightarrow M_{\rm Hod}(\overline{X}).
$$
We need to show that it is compatible with the glueing ${\bf d}$ in the sense that
\begin{equation}
\label{cdcompatible}
C_{{\rm Hod},\overline{X}} \circ {\bf d} = {\bf d}^{-1} \circ C_{{\rm Hod},X}.
\end{equation}
First of all  $C_{{\rm Hod},X}$ and $C_{{\rm Hod},\overline{X}}$ intertwine 
the multiplication action of $\Gm$, with the complex conjugation $\Gm \cong \overline{\mathbb G}_m$. 
Also, $C$ and ${\bf d}$ both fix the de Rham fiber over $\lambda = 1$. Hence, to verify the compatibility
\eqref{cdcompatible}, it suffices to verify it over $\lambda = 1$. Here
$$
C_{{\rm DR},X}: M_{\rm DR}(X) \rightarrow M_{\rm DR}(\overline{X}), \;\;\; (E,\nabla )\mapsto (\overline{E},\overline{\nabla})
$$
and composing with the isomorphism ${\bf d}^{-1}$ which comes from $\pi _1(X,x)\cong \pi _1(\overline{X},\overline{x})$
we get an antilinear automorphism of $M_{\rm DR}(X)$. It is easy to see that, in terms of the isomorphism with $M_B(X)$,
this antilinear automorphism is just the complex conjugation action on representations, $\rho \mapsto \overline{\rho}$
where $\overline{\rho} (\gamma ) := \overline{\rho (\gamma )}$. Similarly, $C_{{\rm Hod},\overline{X}} \circ {\bf d}$ is
also seen to be the same automorphism $\rho \mapsto \overline{\rho}$. This proves the equality \eqref{cdcompatible}. 

With this compatibility, $C_{{\rm Hod},X}$ and $C_{{\rm Hod},\overline{X}}$ glue to give an 
antiholomorphic involution 
$$
C: M_{\rm DH}(X)\rightarrow M_{\rm DH}(X).
$$
covering the involution $\lambda \mapsto \overline{\lambda}^{\, -1}$ of $\pp ^1$.
As described above, on the  fiber over $\lambda = 1$ which is $M_{\rm DR}(X)\cong M_B(X)$, the involution is
$C(\rho ) = \overline{\rho}$.

The dual of a vector bundle with $\lambda$-connection $(E,\nabla )$ is again 
a vector bundle with $\lambda$-connection denoted $(E^{\ast}, \nabla ^{\ast})$. This operation
is compatible with multiplication by $\Gm$, and with the operation of taking the dual
of a local system via the Riemann-Hilbert correspondence. Therefore, it is compatible with the
glueing ${\bf d}$ and gives an involution, holomorphic this time, denoted 
$$
D: M_{\rm DH}(X)\rightarrow M_{\rm DH}(X)
$$
which covers the identity of $\pp ^1$. 

Finally, multiplication by $-1\in \Gm$ gives an involution of $M_{\rm DH}(X)$ denoted by $N$,
covering the involution $\lambda \mapsto -\lambda $ of $\pp ^1$. 

\begin{lemma}
The involutions $C$, $D$ and $N$ commute. Their product is an involution $\sigma$ of 
$M_{\rm DH}(X)$  covering the antipodal involution $\lambda \mapsto -\overline{\lambda}^{\, -1}$
of $\pp ^1$.
\end{lemma}
{\em Proof:}
The complex conjugate of the dual of a vector bundle is naturally isomorphic to the dual of the complex 
conjugate. These also clearly commute with the operation of multiplying the connection by $-1$. Hence, the three
involutions commute, which implies that the product $CDN$ is again an involution. It is antilinear because
$C$ is antilinear whereas $D$ and $N$ are $\cc$-linear. Since $C$, $D$ and $N$ cover respectively the 
involutions $\lambda \mapsto \overline{\lambda}^{\, -1}$, $\lambda \mapsto \lambda$ and $\lambda \mapsto -\lambda$,
their product covers the product of these three, which is the antipodal involution. 
\eop

\subsection{Preferred sections and the twistor property}
\label{sec-preferred-twistor}
Deligne proposed to construct a family of ``preferred sections'' of the glued space $M_{\rm DH}(X)$, one for each harmonic bundle on $X$.

\begin{proposition}[\cite{Hitchin} \cite{DeligneLett} \cite{hfnac}]
\label{prop-pref}
Suppose $(E,\partial , \delbar , \theta , \overline{\theta} )$ is a harmonic bundle on $X$.
Then it leads to a section $\Pp : \pp ^1 \rightarrow M_{\rm DH}(X)$
which is $\sigma$-invariant and which sends $\lambda \in \aaa ^1$ to the 
holomorphic bundle $(E, \delbar + \lambda \overline{\theta})$ with 
$\lambda$-connection $\nabla = \lambda \partial + \theta$. 
\end{proposition}
{\em Proof:}
We first define the value of the section at $\lambda \in \aaa ^1$. 
On the $\Cc ^{\infty}$ bundle $E$, consider the holomorphic structure 
$$
\delbar _{\lambda} := \delbar + \lambda \overline{\theta}.
$$
The holomorphic bundle $E^{\lambda}:= (E,\delbar _{\lambda})$ admits a 
$\lambda$-connection operator $\nabla _{\lambda} := \lambda \partial + \theta$. 
This gives a point $(E^{\lambda}, \nabla _{\lambda})$ in $M_{\rm Hod}(X)_{\lambda}$. 
One checks that over $\Gm \subset \aaa ^1$, this section is invariant under the antipodal involution operator.
Hence, taking  image of the graph of our section already defined over $\aaa ^1$,  by $\sigma$, gives the section over the other chart
$\aaa ^1$ at infinity, and over $\Gm$ these glue together. One can also define directly the value of the section on the complex conjugate chart,
see for example \cite[pp 20-24]{twistor}. 
\eop

In Hitchin's original point of view \cite{HitchinH} \cite{Hitchin}, the twistor space $M_{\rm DH}(X)$ came from the Penrose construction for the 
quaternionic structure on $M(X)$  whose different complex structures were
those of $M_{\rm Dol}$ and $M_{\rm DR}$. 
The Penrose twistor space has a natural product structure of the form $\pp ^1 \times M(X)$.

In Deligne's 
reinterpretation
\cite{DeligneLett}
we can first construct the space $M_{\rm DH}(X)$ using the notion of $\lambda$-connection, complex conjugation and the Riemann-Hilbert correspondence
as described above. The product structure is obtained by considering the family of preferred sections as described in the previous
proposition.  This leads back to the quaternionic structure by looking at the tangent space near a preferred section. The key to this beautiful procedure is
the observation that the relative tangent space, or equivalently the normal bundle, along a prefered section is a semistable bundle of slope $1$ on $\pp ^1$,
which is to say it is isomorphic to $\Oo _{\pp ^1}(1)^{\oplus a}$. This weight one property is equivalent to having a quaternionic structure, as was observed
in \cite{HitchinH}. 

There is an equivalence of categories between quaternionic vector spaces, and vector bundles of slope $1$ over $\pp ^1$ with
involution $\sigma$ covering the antipodal involution. If  $V=\Oo _{\pp ^1}(1)^d$ is a slope one bundle, the
space of sections is $H^0(\pp ^1,V) \cong \cc ^{2d}$. If $\sigma$ is an antipodal involution, the space of $\sigma$-invariant sections is a real
form of the space of sections, thus
$$
H^0(\pp ^1, V)^{\sigma}\cong \rr ^{2d},
$$
and the twistor property says that the map from here to any of the fibers $V_{\lambda}$ is an isomorphism.
This is what provides the single real vector space $\rr ^{2d}$ with a whole sphere of different complex
structures. 

With this equivalence, saying that the various complex structures on $M(X)$ correspond to a quaternionic structure is equivalent
to saying that the normal bundle to a preferred section has slope $1$. One can show directly the weight $1$ property given the
construction of $M_{\rm DH}(X)$ and the preferred sections of Proposition \ref{prop-pref}, see \cite{hfnac}. 
It then follows that the deformation space of a preferred section in 
the world of $\sigma$-invariant sections of the fibration $M_{\rm DH}(X)\rightarrow \pp ^1$, maps isomorphically to the tangent space 
of any of the moduli space fibers (for example $M_{\rm DR}(X)$ over $\lambda = 1$ or $M_{\rm Dol}(X)$ over $\lambda = 0$). 
It implies that, locally, there is a unique $\sigma$-invariant section going through any point, and gives an alternative proof of Hitchin's theorem
that the moduli space has a quaternionic structure. For rank one bundles, this property can be globalized:

\begin{lemma}
\label{uniglobal}
For bundles of rank $1$ on a compact $X$, the evaluation morphism at any point $p\in\pp ^1$ 
$$
\Gamma (\pp ^1,M_{\rm DH}(X))^{\sigma} \rightarrow M_{\rm DH}(X)_p
$$
is an isomorphism. 
\end{lemma}
\begin{proof}
In the rank one case, the moduli space is a Lie group so we can use its exponential exact sequence. The tangent at the identity
preferred section is purely semistable of slope $1$. There is a lattice $A= H^1(X,\zz )\cong \zz ^a$  and a finite group $B= H^2(X,\zz )$ such that we have an exact sequence
$$
0\rightarrow A \rightarrow A\otimes \Oo _{\pp ^1} (1)\rightarrow M_{\rm DH}(X) \rightarrow B \rightarrow 0.
$$
We have $H^1(\pp ^1,A)= H^1(\pp ^1, A\otimes \Oo _{\pp ^1})=H^1(\pp ^1,B)=0$, so taking sections gives an exact sequence. 
The subgroups of $\sigma$-invariants again form an exact sequence. The weight one property, equivalent to the quaternionic structure,
says
$$
\Gamma (\pp ^1, A\otimes \Oo _{\pp ^1} (1))^{\sigma} \stackrel{\cong}{\rightarrow} A\otimes \Oo _{\pp ^1} (1)_p.
$$
Comparing with the exact sequence of values at $p$ gives the result for $M_{\rm DH}(X)$. 
\end{proof}

Deligne gave the construction of a quaternionic structure associated to a weight $1$ real Hodge structure in \cite{DeligneLett}. 
Given a vector space $V$ with two filtrations $F$ and $\overline{F}$, we can form a bundle 
$\xi (V,F,\overline{F})\rightarrow \pp ^1$ and this bundle has slope $1$ if and only if the two filtrations are $1$-opposed, i.e. they define
a Hodge structure of weight $1$. In this sense, the twistor property is analogous to saying that a Hodge structure has weight $1$. 
These slightly different points of view are compatible for a preferred section which comes from a variation of Hodge structure, where the
tangent space to the moduli space  has a natural weight $1$ Hodge structure. 

The bundles of the form $\xi (V,F,\overline{F})$ are the slope $1$ bundles together with additional structure of an action of $\Gm$.
A somewhat different collection of additional structure involving a connection is used in the notion of $tt^{\ast}$ geometry
\cite{CecottiVafa} \cite{Hertling}, which also has its physical roots in Hitchin's twistor space. Schmid mentionned, during his courses on Hodge theory,
the similarity between the equations governing the local structure of variations of Hodge structure, and the monopole or Nahm's equations.
It is interesting that
these objects from physics are
so closely related to variations of Hodge structure, the
analytic incarnation of the idea of motives. This suggests a relationship between physics and motives which might be philosophically compelling.

\section{The twistor space for $X=\pp ^1 - \{ 0 , \infty \} $}
\label{sec-p1}

We would now like to mimic the Deligne-Hitchin construction for a quasiprojective curve.
For simplicity of calculation, let us take the easiest case which is $X:= \pp ^1$ and $D:= \{ 0 , \infty \} $. 
Let $U:=X-D$ and fix $x= 1$ as basepoint in $X$ or $U$.
Then $\pi _1(U, x)\cong \zz ^{\perp}$. A choice of $q=\sqrt{-1}\in Q\subset \cc$ yields a choice of generator $\gamma _0(q)\in \pi _1(U, x)$ going once around the origin,
counterclockwise if $1$ is pictured to the right of the origin and $q$ is pictured above the origin. 
Changing the choice of $q$ changes the generator to its inverse, which is why we get $\zz ^{\perp}$ rather than $\zz$. For the local monodromy transformations,
this yields
$$
\pi _1(U,x) ^{\perp} \cong \zz
$$
with a distinguished generator denoted $\gamma _0$.

Let $z$ denote the standard coordinate on $X$.
A logarithmic $\lambda$-connection on the trivial bundle $E:= \Oo _X$ is of the form
$$
\nabla = \lambda d + \alpha \frac{dz}{z}.
$$
In particular, we can write
$$
M_{\rm Hod}(X,\log D) =  \aaa ^1\times \aaa ^1 = \{ (\lambda , \alpha  )\} .
$$
The first coordinate is the parameter $\lambda$ and the second, the residue parameter $\alpha$.

For $\lambda \neq 0$ a point $(\lambda , \alpha )$ corresponds to the $1$-connection $d+ \lambda ^{-1}\alpha (dz/z)$.
Let $\rho : \pi _1(U,x)\rightarrow \Gm$ be its monodromy representation. 

The local monodromy transformation at the origin (see \S \ref{sub-localmonodromy} above) is 
$$
\rho ^{\perp} (\gamma _0) = {\rm exp}^{\perp}(\alpha /\lambda ) = (\cos (2\pi \alpha /\lambda ), \sin (2\pi \alpha /\lambda ))\in \Gm ^{\perp}.
$$
In order to write the global monodromy, note that any loop 
$\gamma \in \pi _1(U,x)$ can be expressed as a unique function 
$$
\gamma : [0,1]\rightarrow U
$$
such that $|\gamma (t)|=1$ and $\gamma$ proceeds at a uniform speed i.e. $\left| \frac{d\gamma}{dt} \right|$ is constant. 
The function $\gamma$ is then real analytic and extends by analytic continuation to a unique map $\gamma : \cc \rightarrow U$. 
In usual terms choosing $i=\sqrt{-1}\in \cc$, the path $\gamma$ is written as $t\mapsto e^{2\pi i t}$ and this expression
is valid for any $t\in \cc$. 

The global monodromy of our differential equation $d+ \lambda ^{-1}\alpha (dz/z)$ can now be expressed by the formula
$$
\rho (\gamma ) = \gamma (\alpha /\lambda ).
$$

Next, note that $\overline{X} = \pp ^1 - \{ 0,\infty \}$ too, and $\overline{x} = 1$ still, so 
we can write $(\overline{X},\overline{D},\overline{x}) = (X,D,x)$. 
In terms of this equality, $\varphi$ is just the geometric operation of complex conjugation on $U^{\rm top}$. 
Hence,
for any $\gamma \in \pi _1(U,x)$, the complex conjugation map $\varphi$ takes $\gamma$ to the loop 
$$
\varphi (\gamma ) = \overline{\gamma} = \gamma ^{-1}.
$$

\subsection{Computation of the Deligne glueing}
\label{sub-computation}
We also have 
$$
M_{\rm Hod}(\overline{X},\overline{D}) = \aaa ^1\times \aaa ^1 = \{ (\lambda , \alpha )\} 
$$
via the identification between $X$ and $\overline{X}$. In order to compute the Deligne glueing map
${\bf d}$, suppose we start with a point $(\lambda , a)$ in $M_{\rm Hod}(X,\log D)$. This corresponds to 
a monodromy representation 
$\gamma \mapsto \gamma (\alpha /\lambda ) $
as explained above. 
The fact that $\varphi$ interchanges $\gamma$
and $\gamma ^{-1}$ means that, after re-identifying $\overline{X}$ with $X$, the image of this representation
by $\varphi^{\ast}$ is 
$$
\gamma \mapsto \gamma ^{-1}(\alpha /\lambda ).
$$
There is a unique way to lift this to a logarithmic connection, if we want to send the point $\alpha=0$ to the point $\alpha=0$
and keep everything continuous: it is the connection $\Phi = d-\lambda ^{-1}\alpha  (dz/z)$. Finally, in the prescription
for the Deligne glueing we set $\mu := \lambda ^{-1}$ and transform this to a $\mu$-connection $\mu \Phi$.
This yields the point 
$$
{\bf d}(\lambda , \alpha ) = (\mu , \beta ) = (\lambda ^{-1}, -\lambda ^{-2}\alpha ). 
$$
We can now glue the two charts to get the Deligne-Hodge twistor space:
$$
M_{\rm DH}(X,\log D) := M_{\rm Hod}(X,D) \sqcup ^{{\bf d}} M_{\rm Hod}(\overline{X},\log \overline{D}).
$$

\subsection{The weight two property}
\label{sub-weighttwo}
Notice that the glueing map is linear in $\alpha$, so in this case the result is a vector bundle over
$\pp ^1$, in fact it is clearly the bundle $\Oo _{\pp ^1}(2)$ with glueing function $-\lambda ^{-2}$. 
The minus sign will have an effect on the antilinear involution below. 

Suppose $\lambda \mapsto P(\lambda )$ is a polynomial considered as a section of $M_{\rm Hod}(X,\log D)$.
Then its graph is the set of points $(\lambda , P(\lambda ))$ and these correspond to  points
of the form $(\lambda ^{-1}, -\lambda ^{-2}P(\lambda ))$ in $M_{\rm Hod}(\overline{X},\overline{D})$ for 
$\lambda$ invertible. Taking the closure over $\mu = \lambda ^{-1}\rightarrow 0$ yields the
set of points of the form $(\mu , -\mu ^2P(\mu ^{-1}))$. This is a holomorphic section in the $\mu$ chart
if and only if $P$ is a polynomial of degree $\leq 2$. This is one way to see that
$M_{\rm DH}(X,\log D) \cong \Oo _{\pp ^1}(2)$.
The global sections are those which are, in the standard
chart $M_{\rm Hod}(X,\log D)$, polynomials of degree $\leq 2$. 

The main point of the title of this paper is that, since this bundle has slope $2$, it corresponds in some sense
to a Hodge structure of weight $2$. That contrasts with the normal bundle of a preferred section in the compact case
(\S \ref{sec-preferred-twistor}), which has slope $1$.  The weight two behavior is to be expected in the present situation,
by analogy with the usual mixed Hodge theory of open varieties, where $H^1(\pp ^1- \{ 0, \infty \} )$ has a pure Hodge structure of
weight two. 

In the present case, the Deligne-Hitchin space $M_{\rm DH}(X,\log D)$ will again have an involution
$\sigma$ to be calculated below. Given that it is a line bundle of slope $2$,  the space of sections has dimension $3$:
$$
\Gamma (\pp ^1, M_{\rm DH}(X,\log D)) \cong \Gamma (\pp ^1, \Oo _{\pp ^1}(2))\cong \cc ^3,
$$
and the $\sigma$-invariant sections are a real form
$$
\Gamma (\pp ^1, M_{\rm DH}(X,\log D))^{\sigma} \cong \rr ^3.
$$
The map from here to any one of the fibers over $\lambda \in \pp ^1$ will be surjective but have a real one-dimensional 
kernel. It turns out that this real kernel corresponds to the real parameter involved in a parabolic structure,
even though we have seen the existence of this additional real parameter without refering {\em a priori} to
the notion of parabolic structure. 

\subsection{The antipodal involution}
\label{sub-antipodal}
We  now calculate explicitly the involution $\sigma$. Recall that it is a product of the
three involutions $C$, $D$ and $N$. The duality involution is trivial on the underlying bundles
because we are using the trivial bundle: $E^{\ast} = \Oo ^{\ast} = \Oo = E$. 
The connection on $E\otimes E^{\ast}$ should be trivial so we see that 
for $\nabla = d + \alpha (dz/z)$ the dual connection is
$\nabla ^{\ast} = d - \alpha  (dz/z)$. Thus
$$
D(\lambda , \alpha  ) = (\lambda , -\alpha ).
$$
Similarly, by definition 
$$
N(\lambda , \alpha ) = (-\lambda , -\alpha ).
$$
Putting these together gives $DN(\lambda , \alpha ) = (-\lambda , \alpha )$. These are  expressed within a single
chart $M_{\rm Hod}(X,\log D)$. On the other hand, the involution $C$ goes from 
the chart $M_{\rm Hod}(X,\log D)$ to the chart $M_{\rm Hod}(\overline{X},\overline{D})$, and with
respect to these charts it is given by
$$
C(\lambda , \alpha ) = (\overline{\lambda}, \overline{\alpha }).
$$
For $\lambda$ invertible we would like to put this back in the original chart. Recall that a point 
of the form $(\mu , \beta )$ in the chart $M_{\rm Hod}(\overline{X},\overline{D})$ corresponds to
$(\mu ^{-1}, -\mu ^{-2}\beta )$ in the chart $M_{\rm Hod}(X,\log D)$. Thus, within the same chart 
$M_{\rm Hod}(X,\log D)$ and for $\lambda$ invertible, the involution $C$ can be expressed as
$$
C(\lambda , \alpha ) = (\overline{\lambda}^{\, -1}, -\overline{\lambda}^{\, -2}\overline{\alpha }).
$$
Putting these together gives our expression for $\sigma = CDN$ again within the original 
chart and for $\lambda$ invertible:
$$
\sigma (\lambda , \alpha ) = (-\overline{\lambda}^{\, -1}, -\overline{\lambda}^{\, -2}\overline{\alpha }).
$$

We would now like to calculate which are the $\sigma$-invariant sections. Recall from \S \ref{sub-computation}
that a global section of $M_{\rm DH}(X,\log D)$ is, in the first chart, a polynomial $P$ of order $\leq 2$. Thus
we can write our section as
$$
P:\lambda \mapsto (\lambda , a_0 + a_1\lambda + a_2\lambda ^2).
$$
Its graph is the set of image points. The transformed section $P^{\sigma}$ has graph which is the closure of
the set of points of the form 
$$
\sigma P(\lambda ) = \left(
-\overline{\lambda}^{\, -1}, -\overline{\lambda}^{\, -2}\overline{(a_0 + a_1\lambda + a_2\lambda ^2)} 
\right)
\\
= \left( -\overline{\lambda}^{\, -1}, -\overline{a_2} - \overline{a_1}\overline{\lambda}^{\, -1}
-\overline{a_0}\overline{\lambda}^{\, -2} \right) .
$$
Substituting in the above expression $-\overline{\lambda}^{\, -1}$ by $t$, the graph becomes the set of points
of the form 
$$
(t,-\overline{a_2} + \overline{a_1}t - \overline{a_0}t^2).
$$
This is the graph of the polynomial $t\mapsto -\overline{a_2} + \overline{a_1}t - \overline{a_0}t^2$.
Thus, writing our polynomials generically with a variable $u$ we can write 
\begin{equation}
\label{eq-sigma}
\left( a_0 + a_1u + a_2 u^2 \right) ^{\sigma} = 
\left( -\overline{a_2} + \overline{a_1}u - \overline{a_0}u^2 \right) .
\end{equation}

The $\sigma$-invariant sections are the polynomials with $a_2 = - \overline{a_0}$ and
$a_1 = \overline{a_1}$. In other words, an invariant section corresponds to a pair 
$(a,\alpha )\in \rr \times \cc \cong \rr ^3$
with the formula 
\begin{equation}
\label{like-mochizuki}
P(\lambda = \psi (a,\alpha )(\lambda ) := \alpha - a \lambda - \overline{\alpha} \lambda ^2.
\end{equation}
The reader will recognize this as the formula from Mochizuki \cite[2.1.7, p. 25]{Mochizuki}.

\subsection{Gauge transformations}
\label{sub-gauge}
A logarithmic connection is not uniquely determined by its monodromy representation. This situation is
complicated in higher rank, but is understood easily in our case from the fact that the monodromy
associated to a $\lambda$-connection $(\lambda , \alpha )$ is ${\rm exp}^{\perp}( \alpha /\lambda )$, or with a
choice of $i=\sqrt{-1}$ it is $e^{2\pi i\alpha /\lambda }$. If we replace $\alpha $ by
$\alpha -k \lambda$ for $k\in \zz$ we get the same monodromy representation. This process may be viewed as
making the meromorphic gauge transformation $v\mapsto z^{-k} v$ on the bundle $E= \Oo _X$, or equivalently changing the
bundle $E$ to $E(kD_0 - kD_{\infty})$ where $D_0=\{ 0\}$ and $D_{\infty} = \{ \infty \}$. The zeros or poles
of the gauge transformation
at points of $D$ change the residues of the $\lambda$-connection by integer multiples of $\lambda$. 
The sign here and in the definition \eqref{like-mochizuki} of $\psi (a,\alpha )$ comes from the formula \eqref{residueformula}.

In terms of our space $M_{\rm Hod}(X,\log D)$ we have an action of $\zz$ obtained by letting $k\in \zz$
act as $(\lambda , \alpha )\mapsto (\lambda , \alpha - \lambda k)$. This action extends to the other chart, 
and gives an action of $\zz$ on $M_{\rm Hod}(X,\log D)$. Over $\Gm \subset \pp ^1$ the
action is discrete and the quotient is $\Gm \times \cc / \zz = \Gm \times M_B(U)$. Note that the
action degenerates to a trivial action on the fibers over $0$ and $\infty$, the quotients of these actions
are trivial $B\zz$-gerbs over $M_{\rm Dol}(X,\log D)$ and $M_{\rm Dol}(\overline{X},\log \overline{D})$. 

The $\zz$ action respects the involution $\sigma$ so it gives an action on the space of $\sigma$-invariant
sections. In terms of the previous formulae, this clearly acts on the degree $1$ term in the polynomials,
or in terms of the coordinates $(a,\alpha )\in \rr \times \cc$ it acts with generator $(1,0)$. 
Thus, we can write
$$
\frac{\Gamma (\pp ^1, M_{\rm DH}(X,\log D))^{\sigma}}{\zz} \cong \frac{\rr \times \cc }{(1,0)\cdot \zz}.
$$
Given a harmonic bundle on $U$ we
get a $\sigma$-invariant preferred section, $a$ is the parabolic weight of the Higgs bundle and $\alpha$ is the
residue of the Higgs field, see Theorem \ref{harmoniccompose} below. 

We recover in this way the space of possible residues of parabolic $\lambda$-connections, with the
residue of the $\lambda$-connection being given by the previous formula \eqref{like-mochizuki}. Note that the action of $\zz$ is by
meromorphic gauge transformations, so moving the parabolic index once around the circle induces an
elementary transformation of the bundle.

\section{The Tate twistor structure}
\label{sec-tatetwistor}

Before getting to the general rank one case, we investigate the structures associated to the bundle $\Oo _{\pp ^1}(2)$
which occurs above. Recall that $T\pp ^1 \cong \Oo _{\pp ^1}(2)$. Furthermore, the sign $-\lambda ^2$ which occurs in the
glueing function for residues of points in $M_{\rm DH}(X,\log D)$ is the same as in the glueing function
for $T\pp ^1$.  The Tate motive is a pure Hodge structure of type 
$(1,1)$ hence weight $2$. In view of this, we define the {\em additive Tate twistor structure} to be the bundle 
$$
T(1):= T\pp ^1 ,
$$
with its natural antilinear involution
$$
\sigma _{T(1)} := \sigma _{T\pp ^1}.
$$ 
See also Mochizuki \cite[\S 3.10.2]{Mochizuki} and Sabbah \cite[\S 2.1.3]{Sabbah}. 

On the other hand, we define the {\em logarithmic Tate twistor structure}
to be the same bundle $T(1,\log ):= T\pp ^1$, but here
the antilinear involution should have a sign change with respect to the natural one on $T\pp ^1$,
$$
\sigma _{T(1,\log )} := - \sigma _{T\pp ^1}.
$$
The reason for the difference between these two will be explained below. 

\subsection{Integer subgroups}
\label{sub-tate-integer}
The action of $\Gm$ on $\pp ^1$ preserving $0$ and 
$\infty$ gives an action of $\Gm$ on $T\pp ^1$. The derivative of this action is a section of the tangent bundle,
defining the integer subgroup $\zz \times \pp ^1 \subset T\pp ^1$. This section is antipreserved by the standard involution
$\sigma _{T\pp ^1}$. 
For the additive Tate twistor structure, we therefore use the imaginary
version of this integer subgroup, the set of integer multiples of $\pm 2\pi \sqrt{-1}$ denoted
$$
\zz (1)\cong \zz ^{\perp} \subset \Gamma (\pp ^1,T(1))^{\sigma}.
$$
For the logarithmic Tate twistor structure, we use the integer subgroup itself
$$
\zz (1,\log ):= \zz \subset \Gamma (\pp ^1,T(1,\log ))^{\sigma}.
$$
Take for generator of $\zz (1,\log )$ the vector field $-\lambda \frac{\partial}{\partial \lambda}$ in the standard chart $\aaa ^1$. 
This minus sign comes from the formula \eqref{residueformula}, see also \S \ref{sub-gauge} above. If we go into the other
chart then the generator changes sign once again. See the paragraph above equation \eqref{dhgauge} for a reflection of this
sign change in the action of the gauge group on the complex conjugate chart. 

The sign change for $\sigma$ on $T(1,\log )$ guarantees that the integer sections are preserved by $\sigma$. 

The action of $\Gm$ on the additive Tate twistor structure $T(1)$ 
corresponds to the usual Tate Hodge structure of type $(1,1)$ with its integral subgroup $\zz (1)$. 

Define the {\em multiplicative Tate twistor structure} to be 
\begin{equation}
\label{gm1}
\Gm (1) := T(1,\log ) / \zz (1,\log ).
\end{equation}
The fiber over $\lambda = 1$, the de Rham version, is naturally identified with $\cc / \zz$. The exponential map gives the isomorphism
$$
{\rm exp}^{\perp} : \cc / \zz \stackrel{\cong}{\rightarrow} \Gm ^{\perp}.
$$
This explains why we needed to change the sign of $\sigma$ for the logarithmic Tate structure: the exponential map interchanges
objects before and after $(\; )^{\perp}$. Thus, to get a Deligne-Tate type twist on the multiplicative group corresponding
to the local monodromy operator of a connection, we need to undo this twist which occurs naturally in $T(1)=T\pp ^1$. 

Over $\lambda = 0$ and $\lambda = \infty$, the quotient defining $\Gm (1)$ is to be taken in the stack sense. Hence,
$$
\Gm (1)_{\rm Dol} = \cc \times B\zz .
$$

\subsection{The antipodal involution in the additive case}
\label{sub-tangent-antipodal}

Start by  computing the natural antipodal involution $\sigma _{T\pp ^1}$ of
the tangent bundle or equivalently the additive Tate structure $T(1)$. For this subsection, the notation 
$\sigma$ represents $\sigma _{T\pp ^1}= \sigma _{T(1)}$.

The vector field $\lambda \frac{\partial}{\partial \lambda}$ goes radially outward from $0$ towards
$\infty$. Up to a scalar it is the unique vector field with zeros at $0$ and $\infty$, and this property is preserved by $\sigma$. 
Geometrically we see that the antipodal involution changes the sign of this radial vector field. Acting on this section considered as
a section of $T\pp ^1$ we get
$$
\sigma ^{\ast}\lambda \frac{\partial}{\partial \lambda} = -\lambda \frac{\partial}{\partial \lambda} .
$$
Similar geometric consideration shows that $\sigma$ interchanges the vector fields $\frac{\partial}{\partial \lambda}$ and 
$\lambda ^2\frac{\partial}{\partial \lambda}$, this time with no sign change. Thus
$$
\sigma ^{\ast}\frac{\partial}{\partial \lambda} = \lambda ^2\frac{\partial}{\partial \lambda} 
$$
and vice-versa. 

A point of the total space of the tangent bundle, in the first standard chart,  has the form $ \left( \lambda , v\frac{\partial}{\partial \lambda} \right)$.
We know that $\sigma$ acts on the first coordinate by sending $\lambda$ to $-\overline{\lambda} ^{\, -1}$.  Thinking of the above sections as
corresponding to their graphs which are sets of points, and noting that $\sigma$ is antilinear in the coordinate $v$,
the formula for $\sigma$ on points of the total bundle is 
$$
\sigma \left( \lambda , b\frac{\partial}{\partial \lambda} \right) = 
\left( -\overline{\lambda} ^{\, -1} , \overline{\lambda} ^{\, -2}\overline{v} \frac{\partial}{\partial \lambda} \right) .
$$

\subsection{The antipodal involution in the logarithmic or multiplicative case}
\label{sub-tate-antipodal}

Recall that $\sigma _{T(1,\log )}= - \sigma _{T\pp ^1}$, with the minus sign acting only in the bundle fiber direction. Hence, for 
$\sigma = \sigma _{T(1,\log )}$ the formulae from the previous section become
$$
\sigma \left( \lambda , v\frac{\partial}{\partial \lambda} \right) = 
\left( -\overline{\lambda} ^{\, -1} , -\overline{\lambda} ^{\, -2}\overline{v} \frac{\partial}{\partial \lambda} \right) ,
$$
and
$$
\sigma ^{\ast}(u+v\lambda + w\lambda ^2)\frac{\partial}{\partial \lambda}
= (-\overline{w}+ \overline{v}\lambda -\overline{u}\lambda ^2)
\frac{\partial}{\partial \lambda} .
$$
This fits with the formula \eqref{like-mochizuki}: a $\sigma$-invariant section has the form
$$
\psi (a,\alpha ) = \lambda \mapsto (\alpha - a \lambda - \overline{\alpha} \lambda ^2)\frac{\partial}{\partial \lambda} 
$$
with $\alpha \in \cc$ and $a\in\rr$. 

On the quotient \eqref{gm1} $\Gm (1) = T(1,\log )/\zz (1,\log ) $ we get the involution $\sigma _{\Gm (1)}$.

\subsection{The space of invariant sections}
\label{sub-invariant}
The space of $\sigma$-invariant sections of $T(1,\log )$ inherits some canonical structure. 
For any point $p\in \pp ^1$ we get a distinguished $\sigma$-invariant direction in 
$\Gamma (\pp ^1,T(1,\log ))$: the sections having simple zeros at $p$ and $\sigma (p)$. 

This space $\Gamma (\pp ^1, T(1,\log )(-p-\sigma (p)))^{\sigma}$
is naturally identified with the Lie algebra of the one parameter group of $\sigma$-antipreserving
homotheties of $\pp ^1$ which fix $p$ and $\sigma (p)$. 
One must say ``antipreserving'' here because of the sign change in $\sigma _{T(1,\log )}$. 

The elements of this group are radial homotheties; the group is isomorphic to $\Gm (\rr )$ and its Lie algebra is isomorphic to $\rr$. 
In particular, there is a distinguished generator which is the vector field going inward from  $\sigma (p)$ towards $p$ attaining speed $1$
at the equator between the two fixed points.  Equivalently, we can require that the expansion factor at $p$ be equal to $-1$. 
This is normalized so that when $p=0$ it gives the generator of 
$\zz (1,\log )$.
The expansion factor of a vector field with a zero, is a well-defined complex number: it is dual to the residue, and
can be defined as the value of the vector field on the differential form $\frac{dz}{z}$. 

Let $\nu ^p\in \Gamma (\pp ^1, T(1,\log )(-p-\sigma (p)))^{\sigma}$ denote the generator normalized to have expansion factor $-1$ at $p$. 
We get a canonical isomorphism 
$$
 \Gamma (\pp ^1, T(1,\log )(-p-\sigma (p)))^{\sigma} \cong \rr , \;\;\; \nu _p \mapsto 1 .
$$
Evaluating at $p$ gives a map ${\rm ev}_p$ from the space of sections to the fiber $T(1,\log )_p$. 

\begin{lemma}
\label{exactatp}
These maps fit into a canonical exact sequence depending on $p\in\pp ^1$,
$$
0\rightarrow \rr \stackrel{\nu_p}{\rightarrow} \Gamma (\pp ^1,T(1,\log ))^{\sigma} \stackrel{{\rm ev}_p}{\rightarrow} T(1,\log )_p \rightarrow 0.  
$$
\end{lemma}
\begin{proof}
It is exact in the middle because $\rr \cdot \nu _p$ is exactly the space of sections vanishing at $p$.
Exactness on the left and right follow by dimension count. 
\end{proof}

\subsection{The residue evaluation}
\label{sub-reseval}
The standard translation action of ${\mathbb G}_a$ on $\pp ^1$ fixing the point $\infty$ gives a trivialization
$$
T(1,\log )|_{\aaa ^1}\cong \Oo .
$$
For any point $p\in \aaa ^1$,
let ${\rm res}_p$ denote the composition of this trivialization at $p$, with the evaluation map ${\rm ev}_p$. 
Then the exact sequence \ref{exactatp} can be written
\begin{equation}
\label{eq-exactatp}
0\rightarrow \rr \stackrel{\nu_p}{\rightarrow} \Gamma (\pp ^1,T(1,\log ))^{\sigma} \stackrel{{\rm res}_p}{\rightarrow} \cc \rightarrow 0.
\end{equation}
As calculated above, the $\sigma$-invariant 
sections are identified with the polynomials of the form 
$$
\psi (a,\alpha )= \left( \alpha - a \lambda - \overline{\alpha} \lambda ^2\right) \frac{\partial}{\partial \lambda}
$$
with $a\in \rr$ and $\alpha \in \cc$. 
The residue evaluation at $p$ is 
\begin{equation}
\label{for-reseval}
{\rm res}_p(\psi (a,\alpha )) = \alpha - ap -\overline{\alpha}p^2.
\end{equation}
Notice that ${\rm res}_p(\psi (a_p,\alpha _p))=0$, since $\nu _p$ is a section vanishing at $p$.

\subsection{The generator $\nu _p$}
\label{sub-onedim}
Suppose given a point $p\in \aaa ^1$. In coordinates, $\sigma (p) = -\overline{p}^{\, -1}$. 
The
line of $\sigma$-invariant sections which vanish to first order at $p$  is 
$$
\alpha - a p - \overline{\alpha} p^2  = 0.
$$
Vanishing at $\sigma (p)$ is a consequence, because a $\sigma$-invariant
section vanishing at $p$ also has to vanish at $\sigma (p)$. 
Recall that $a\in \rr$. We get a real one-dimensional space of solutions generated for example by
$(a,\alpha )$ with
$$
a = 1- |p|^2, \;\;\; \alpha  = p.
$$ 
Then 
$$
\psi (a,\alpha ) (\lambda )= \left( p + (|p|^2-1)\lambda - \overline{p}\lambda ^2\right)  \frac{\partial}{\partial \lambda}.
$$
Let us calculate the expansion factor at $p$ of the vector field corresponding
to our section $\psi (a_1,\alpha _1)$. For this, express the vector field in the form 
$$
\psi (a,\alpha ) = \eta \cdot (\lambda -p)\frac{\partial}{\partial\lambda} + \ldots
$$
where the $\ldots $ signify higher order terms at $p$. The constant $\eta$ is the expansion factor. 
We have 
$$
\eta = \frac{d}{d\lambda} \left.  \left(  p + (|p|^2-1)\lambda - \overline{p}\lambda ^2 \right) \right| _{\lambda = p} 
$$
$$
= (|p|^2 -1 - 2\overline{p}\lambda )|_{\lambda = p} = -(1+|p|^2)
$$
We can normalize to get the canonical generator $\nu _p = \psi (a_p,\alpha _p)$ whose expansion factor is $-1$:
\begin{equation}
\label{generator-one}
a_p = \frac{1-|p|^2}{1+|p|^2}, \;\;\; \alpha _p = \frac{p}{1+|p|^2}.
\end{equation}
For $p=1$ it is the generator $(1,0)$ of $\zz (1,\log )$.

\subsection{A natural inner product}
\label{sub-innerprod}
For all the above vectors \eqref{generator-one}, we have
$a_p^2 + 4|\alpha _p|^2 = 1$. 
The equation $a^2 + 4|\alpha |^2 = 1$ defines an $S^2\subset \rr \times \cc$,
and the function $p\mapsto (a_p,\alpha _p)$ provides an isomorphism between $\pp ^1$ and this $S^2$. 
This is the unit sphere for the inner product
\begin{equation}
\label{scalarprod}
(a,\alpha )\cdot (b,\beta):= ab + 2( \alpha \overline{\beta} + \overline{\alpha}\beta ).
\end{equation}

\begin{lemma}
\label{so3-innerprod}
The group $SO(3)$ acts naturally as the group of $\sigma$-intertwining complex automorphisms of $\pp ^1$,
so it acts on the space of sections $\Gamma (\pp ^1,T(1,\log ))^{\sigma}$. The above inner product \eqref{scalarprod}
is the $SO(3)$-invariant one, unique up to a scalar.
\end{lemma}
\begin{proof}
Any $\sigma$-intertwining automorphism $f$ of $\pp ^1$ acts on $T(1,\log )$ because that bundle
is naturally defined as the tangent bundle with a twisted $\sigma$. Hence it acts on the space of $\sigma$-invariant
sections. The section $\nu _p$ is canonically defined depending on the point $p$ and the
involution $\sigma$. Hence $f_{\ast}$ takes $\nu _p$ to $\nu _{f(p)}$. It follows that the action of $f$ preserves the
sphere $S^2$ image of the map $\nu$. Therefore $f$ is in the orthogonal group $O(3)$ for this scalar product. It has
determinant $1$ because of holomorphicity. We get $f\in SO(3)$, the special orthogonal group for the scalar
product \eqref{scalarprod}.  
\end{proof}

\subsection{The parabolic weight function}
\label{sub-parabolicweight}

The space of $\sigma$-invariant sections of $T(1,\log )$ is an $\rr ^3$, and at any $p\in \pp ^1$
it is naturally an extension of $T(1,\log )_p$ by $\rr$ (Lemma \ref{exactatp}). The quotient $T(1,\log )_p$ represents the residue of a $\lambda$-connection.
The extra real parameter corresponds to the real parabolic weight of a parabolic structure. However, we need to discuss the
normalization of this identification splitting the exact sequence. 

We use the coordinates $a,\alpha$ for the set of invariant sections denoted $\psi (a,\alpha )$, giving an isomorphism between this space and $\rr \times \cc$.

The point $(1,0)\in \rr \times \cc$ is the generator of the subgroup $\zz (1, \log )$. For each $p$ we have the point
$\nu _p = \psi (a_p,\alpha _p)$ given by \eqref{generator-one}. The quotient of $\rr \times \cc$ by the line generated by 
$(a_p,\alpha _p)$, is the space of residues at $p$. 

We would like to define a {\em parabolic weight function} $\varpi _p : \rr \times \cc \rightarrow \rr$, depending on the point $p$,
such that $\varpi _p(1,0) = 1$ for compatibility with local gauge transformations; and $\varpi _p(a_p,\alpha _p) \neq 0$
so that the local residue map is an isomorphism between $\ker (\varpi _p)$ and $\cc$.

Using the inner product $(a,\alpha )\cdot (a',\alpha ') = aa' + 2( \alpha \overline{\alpha}' + \overline{\alpha}\alpha ')$,
the simplest thing to do is to  
let $\varpi _p$ be given by the inner product with the average of the two vectors $(1,0)$ and $(a_p,\alpha _p)$, then normalize to get
$\varpi _p(1,0)=1$. This is
\begin{eqnarray}
\label{for-paraweight}
\varpi _p(a,\alpha )& := &\frac{(a, \alpha )\cdot (1,0) + (a,\alpha )\cdot (a_p,\alpha _p)}{(1,0)\cdot (1,0) + (1,0)\cdot (a_p,\alpha _p)}
\nonumber \\
& = & \frac{(1+|p|^2)a + (1-|p|^2)a + 2 (\alpha \overline{p} + \overline{\alpha}p)}{(1+|p|^2) + (1-|p|^2)}
\nonumber \\
& =& a +  \alpha \overline{p} +\overline{\alpha}p .
\end{eqnarray}

The parabolic weight and residue functions are the same as Mochizuki's functions ${\mathfrak p}$ and ${\mathfrak e}$ of \cite[\S 2.1.7]{Mochizuki},
however we have preferred to motivate their introduction independently above. 

\begin{proposition}
\label{prop-parabolicweight}
For any point $p\in \aaa ^1$, the parabolic weight function and the residue give an isomorphism
$$
(\varpi _p, {\rm res}_p): \Gamma (\pp ^1,T(1,\log ))^{\sigma} \stackrel{\cong}{\rightarrow} \rr \times \cc .
$$
Let $\zz = \zz (1,\log )$ act on $\rr \times \cc$ with generator $(1,-p)$ in keeping with \eqref{residueformula}. 
Then the above isomorphism descends to the quotient to give
$$
(\varpi _p, {\rm res}_p): \frac{\Gamma (\pp ^1,T(1,\log ))^{\sigma}}{\zz (1,\log )} \stackrel{\cong}{\rightarrow} \frac{\rr \times \cc }{(1,-p)\zz}.
$$
In terms of the coordinates $(a, \alpha )$ given by the construction $\psi$, we have
$$
(\varpi _p, {\rm res}_p)(\psi (a,\alpha )) = (a +  \alpha \overline{p} +\overline{\alpha}p, \alpha - ap -\overline{\alpha}p^2).
$$
\end{proposition}
\begin{proof}
We have chosen $\varpi _p$ so that $\varpi _p (\nu _p) = 1$. From the exact sequence of Lemma \ref{exactatp}, 
this implies that $\varpi _p$ and ${\rm res}_p$ are linearly independent so by dimension count we get the first isomorphism. 
For the second part, it suffices to recall that the subgroup $\zz (1,\log )\subset \Gamma (\pp ^1,T(1,\log ))^{\sigma}$ is given by 
generator $\psi (1,0)$, and to note that
$$
(\varpi _p,{\rm res}_p)(\psi (1,0)) = (1, -p).
$$ 
The formula in terms of coordinates just recalls our calculations above. 
\end{proof}

\subsection{Some questions}
\label{sub-questions}

The residue of a 
logarithmic $\lambda$-connection at a singular point lies in the fiber  $T(1,\log )_{\lambda}$. This
was seen by direct calculation: using the standard frame for $T(1,\log )$ over $\aaa ^1$ and the expression of the residue as a
well-defined complex number, we obtain this identification over the chart $\aaa ^1$. Then by calculation, it is compatible
with the corresponding identification for $\overline{X}$ over
the chart at infinity,  using the Riemann-Hilbert correspondence. This passage by an explicit calculation is unsatisfactory
but I haven't seen any way of improving it, so we formulate a question: 

\begin{question}
\label{quest-twistorp1}
Is there some more natural geometric way of identifying the residue of a logarithmic $\lambda$-connection at a singular
point, with a tangent vector to the $\lambda$-line? 
\end{question}

One possible approach  would be to give a geometric description of the meaning of points in the twistor line
$\pp ^1$. 

Similarly, we picked the definition of the parabolic weight function ``out of the hat''. Taking the scalar product with the average
of the two vectors, then normalizing, is certainly the easiest way to make sure that the function takes on nonzero values on the two vectors,
furthermore the resulting formula for $\varpi _p$ is relatively simple. Nonetheless, it would be better to  have a more motivated
reason for this choice. 

\begin{question}
\label{quest-paraweight}
Is there a geometric interpretation of the meaning of the parabolic weight function, preferably going with the geometric interpretation
we are looking for in Question \ref{quest-twistorp1}?
\end{question}

Another direction of questions is the relationship with $SO(3)$. 
The group of $\sigma$-invariant automorphisms of $\pp ^1$ is the group of metric automorphisms of
$S^2$, in other words it is $SO(3)$ by Lemma \ref{so3-innerprod}. The space of $\sigma$-invariant sections of $T(1,\log )$, which are the
$\sigma$-antiinvariant sections of the tangent bundle, may be 
identified with the perpendicular of the Lie algebra ${\mathfrak s}{\mathfrak o}(3)^{\perp}$. 
By naturality, $T(1,\log )$ also has an action of
$SO(3)$, corresponding to the adjoint action on the Lie algebra. 

\begin{question}
\label{quest-so3}
What is the significance of this action of $SO(3)$ and the identification of elements of the space of parabolic weights and residues,
with vectors in  the Lie algebra? 
\end{question}

It seems to be one of the subjects of Gukov and Witten's paper \cite{GukovWitten}.

\section{The general rank one case}
\label{sec-general}

Consider now the following situation: $X$ is a smooth projective variety, and
$D\subset X$ is a reduced strict normal crossings divisor written as $D=D_1 + \ldots + D_k$
where $D_i$ are its distinct smooth connected irreducible components. Let $U:= X-D$.

\subsection{The Hodge moduli space}
\label{sec-hodge}
Let $M_{\rm Hod}(X,\log D)$ denote the moduli space of triples $(\lambda , L , \nabla )$ where
$\lambda \in \aaa ^1$, $L$ is a line bundle on $X$ such that 
\begin{equation}
\label{c1cond}
c_1(L)_{\qq} \in \qq \cdot [D_1] + \cdots + \qq \cdot [D_k] \subset H^2(X, \qq ^{\perp}),
\end{equation}
and 
$$
\nabla : L\rightarrow L\otimes _{\Oo _X}\Omega ^1_X(\log D)
$$
is an integrable logarithmic $\lambda$-connection on $L$. 
The first coordinate is a map 
$$
\lambda : M_{\rm Hod}(X,\log D)\rightarrow \aaa ^1.
$$

Let ${\rm res}(\nabla ; D_i )\in \cc $ denote the residue of $\nabla$ along $D_i$.
Recall that the residue is a locally constant function and $D_i$ is connected so it is a complex scalar. 

For $\lambda \neq 0$, we have 
\begin{equation}
\label{c1compatible}
\lambda c_1(L) = - \sum _i {\rm res}(\nabla ; D_i)\cdot [D_i] \;\; \mbox{in} \;\; H^2(X, \cc ^{\perp}),
\end{equation}
as can be calibrated by comparing with the gauge transformation formulae \eqref{c1formula} and \eqref{residueformula}.
So the condition about $c_1(L)$ in the definition of $M_{\rm Hod}(X,\log D)$ is automatically satisfied when $\lambda \neq 0$, however
for $\lambda = 0$ this condition is nontrivial. 

Tensor product gives $M_{\rm Hod}(X,\log D)$ a structure of abelian group scheme relative to $\aaa ^1$. We need to use the
condition about Chern classes to prove that it is smooth over $\aaa ^1$, otherwise there would be additional irreducible 
components lying over $\lambda = 0$. 

\begin{lemma}
The morphism $M_{\rm Hod}(X,\log D)\rightarrow \aaa ^1$ is smooth.
\end{lemma}
{\em Proof:}
Suppose $\phi : Y\rightarrow M_{\rm Hod}(X,\log D)$ is a morphism from an artinian scheme. Suppose $Y\subset Y'$ is
an artinian extension provided with a morphism $\lambda ' : Y'\rightarrow \aaa ^1$. We need to extend to 
$Y'\rightarrow M_{\rm Hod}(X,\log D)$ lifting $\lambda '$. The map $\phi$ corresponds to a 
line bundle with integrable $\lambda$-connection $(L,\nabla )$ on $X\times Y$. By smoothness of the Picard scheme of $X$,
this extends to a line bundle $L'$ on $X\times Y'$. The condition about the Chern class of $L$ implies that there
exists some logarithmic connection $\nabla _{1,y}$ on $L_y$ where $y\in Y$ denotes the closed point and $L_y$ is the 
restriction of $L$ to the fiber over $y$. By smoothness of $M_{DR}(X,\log D)$, which follows because of its group structure under tensor
product, we can extend $\nabla _{1,y}$ to
some integrable connection $\nabla '_1$ on $L'$. Then $\lambda ' \nabla '_1$ is an integrable $\lambda '$-connection on 
$L'$. Restricted to $X\times Y$, we can write
$$
\nabla = \lambda ' \nabla '_1 |_{X\times Y} + A
$$
where 
$$
A\in H^0(X\times Y, \Omega ^1_X(\log D)\otimes _{\Oo _X}\Oo _{X\times Y}) \cong H^0(X,\Omega ^1_X(\log D)) \otimes _{\cc}\Oo _Y .
$$
Now extend $A$ in any way to a section
$$
A' \in H^0(X\times Y', \Omega ^1_X(\log D)\otimes _{\Oo _X}\Oo _{X\times Y'}) \cong H^0(X,\Omega ^1_X(\log D)) \otimes _{\cc}\Oo _{Y'} 
$$
and set
$$
\nabla ' := \lambda ' \nabla '_1  + A'.
$$
This provides the required extension. 
\eop

\subsection{Gauge group action}
\label{sec-gauge}
Recall the action of the local meromorphic gauge group $\Gg := \zz ^k$ on $M_{\rm Hod}(X,\log D)$. A vector
$g = (g _1, \ldots , g _k)$ sends $(\lambda , L , \nabla )$ to 
$(\lambda , L (g_1D_1 + \cdots + g _kD_k), \nabla ^{\alpha})$ where $\nabla ^{\alpha}$ is the logarithmic
$\lambda$-connection on $L (\alpha _1D_1 + \cdots + \alpha _kD_k)$ which coincides with $\nabla$ over $U$, 
via the isomorphism 
$$
L (g _1D_1 + \cdots + g _kD_k)|_U \cong L|_U.
$$

We have \eqref{residueformula}
$$
{\rm res}(\nabla ^{g}  ; D_i)  = {\rm res}(\nabla  ; D_i)  - \lambda g _i .
$$

The vector of residues, viewed in the standard framing $\frac{\partial}{\partial\lambda}$ for $T(1,\log )$, provides a morphism 
$$
R : M_{\rm Hod}(X,\log D)\rightarrow T(1,\log ) ^k 
$$
where the $i$-th coordinate of $R(\lambda , L , \nabla )$ is by definition ${\rm res}(\nabla  ; D_i)\cdot \frac{\partial}{\partial\lambda}$. 

The morphism $R$
is compatible with the action of $\Gg = \zz ^k = \zz (1,\log )^k$, where $g \in \zz ^k$ acts on $T(1,\log )^k$ by 
$$
(v_1,\ldots ,v_k)\mapsto (v_1-g_1\lambda \frac{\partial}{\partial\lambda}, \ldots , v_k-g_1\lambda \frac{\partial}{\partial\lambda}),
$$ 
adding $g$ times our standard generator of $\zz (1,\log )$. 

Let 
$$
M_{\rm Hod}(X,\log D)_{\Gm} := M_{\rm Hod}(X,\log D)\times _{\aaa ^1}\Gm .
$$
Then $\Gg$ acts properly discontinuously on $M_{\rm Hod}(X,\log D)_{\Gm}$ because this action lies
over the proper discontinuous action on $T(1,\log )^k_{\Gm}$ via the map $R$.
In the analytic category, we can form the quotient, and the Riemann-Hilbert correspondence gives
an isomorphism
$$
M_{\rm Hod}(X,\log D)_{\Gm}^{\rm an}/\Gg \cong \Gm \times M_B(U)
$$
where $M_B(U):= Hom (\pi _1(U), \Gm )$. 

In view of the Riemann-Hilbert correspondence, we define $M_{\rm Hod}(U)$ to be the stack-theoretical quotient
$$
M_{\rm Hod}(U):= M_{\rm Hod}(X,\log D)/\Gg ,
$$
and similarly for the fibers over $\lambda = 0,1$:
$$
M_{\rm Dol}(U):= M_{\rm Dol}(X,\log D)/\Gg , \;\;\; 
M_{\rm DR}(U):= M_{\rm DR}(X,\log D)/\Gg .
$$
Note that $\Gg$ acts trivially on $M_{\rm Dol}(X,\log D)$ so the quotient $M_{\rm Dol}(U)$ is a stack with $\Gg$ in the general stabilizer group.
If we started with a stack version of $M_{\rm Dol}(X,\log D)$ then the general stabilizer also contains the automorphism group $\Gm$ of a rank one Higgs bundle.
In the rank one case, the stabilizer groups are all the same. So the full stabilizer group of any point of $M_{\rm Dol}(U)$
would be $\Gm \times \Gg$. 

Using these definitions, the RH correspondence again says 
$M_{\rm DR}(U)\cong M_B(U)$, and the 
Deligne glueing process  applies as in \S \ref{sec-glue} to give an analytic stack 
$$
M_{\rm DH}(U) \rightarrow \pp ^1
$$
whose charts are $M_{\rm Hod}(U)$ and $M_{\rm Hod}(\overline{U})$. Note however that these charts don't have algebraic structures. 

We would like to investigate how to lift to a Deligne-Hitchin glueing on the space of logarithmic connections,
to get an analytic stack $M_{\rm DH}(X,\log D)$
which would have nicer geometric properties---its charts would be the Artin algebraic stacks.
We would then have a quotient expression
$$
M_{\rm DH}(U) = M_{\rm DH}(X,\log D)/\Gg \rightarrow \pp ^1 .
$$
One way of looking at this question would be to calculate the fundamental group of 
$M_{\rm DH}(U)\rightarrow \pp ^1 $ and see if it has a covering which resolves the stackiness over
$0$ and $\infty$. Instead, we construct directly the covering.

\subsection{The Riemann-Hilbert correspondence and glueing}
\label{sec-rhglue}
Our goal in this subsection is to define $M_{DH}(X,\log D)$ by Deligne glueing of $M_{Hod}(X,\log D)$ with
$M_{Hod}(\overline{X}, \log \overline{D})$. 

It will be useful to have a Betti version of $M_{DR}(X,\log D)$ to intervene in the
glueing. 
Suppose $\rho \in M_B(U)$. Recall from \S \ref{sub-localmonodromy} that for each component $D_i$ of the divisor, 
we get a well-defined local monodromy element $\rho ^{\perp} (\gamma _{D_i})\in \Gm ^{\perp}$. 

Consider the following diagram:
$$
M_B(U) \rightarrow (\Gm ^{\perp}) ^k \stackrel{{\rm exp}^{\perp}}{\longleftarrow} \cc ^k,
$$
where the $k$ copies are for the $k$ components of the divisor $D= D_1 + \ldots + D_k$,
the first map sends $\rho$ to its vector of local monodromy transformations,
and 
$$
{\rm exp}^{\perp}:(a_1,\ldots , a_k)\mapsto \left( (\cos (2\pi a_1), \sin (2\pi a_1)),\ldots , (\cos (2\pi a_k), \sin (2\pi a_k)) \right) .
$$ 
Let $M_B(X, \log D)$ denote the fiber product. Thus, a point in $M_B(X,\log D)$ is an uple
$(\rho ; a_1,\ldots , a_k )$ where $\rho$ is a representation of rank one over $U$ and
$a_i\in \cc$ are choices of circular logarithms of the monodromy operators $\rho ^{\perp}(\gamma _{D_i})\in \Gm ^{\perp}$.

Define an action of the gauge group $\Gg = \zz ^k$ on $M_B(X,\log D)$ by
$$
g= (g_1,\ldots , g_k): (\rho ; a_1,\ldots , a_k)\mapsto (\rho ; a_1 - g_1, \ldots , a_k - g_k).
$$

\begin{lemma}
\label{rhlift}
The Riemann-Hilbert correspondence lifts to
$$
M_{DR}(X,\log D) \cong M_B(X,\log D).
$$
\end{lemma}
{\em Proof:}
Given a line bundle with integrable connection $(L,\nabla )$, associate the 
point 
$$
(\rho , {\rm res}(\nabla ;D_1),\ldots , {\rm res}(\nabla ; D_k)) \in 
M_B(U)\times _{(\cc ^{\ast})^k} \cc ^k = M_B(X,\log D).
$$
This is a morphism of analytical groups with the same dimension, so it suffices to prove that it is
injective and surjective. Before doing that, we verify compatibility with the gauge group action. 
Given $g= (g_1,\ldots , g_k)\in\Gg $, notice that the monodromy representation of $(L^g,\nabla ^g)$ is
the same as $\rho$. For the circular logarithms, the formula for residues 
$$
{\rm res}(\nabla ^g; D_i) = {\rm res}(\nabla ; D_i) -g_i
$$
implies the required compatibility with the gauge group action. 

For injectivity, suppose $(L,\nabla )$ and $(L',\nabla ')$ are two 
line bundles with connection, with the same monodromy representation and the same residues. 
There is a unique isomrphism $\psi : L|_{U}\cong L'|_{U}$ compatible with the monodromy representation
or equivalently with $\nabla$ and $\nabla '$ on $U$. 
Then, the poles or zeros of $\psi$ along a component $D_i$ are determined by the difference between the 
residues of $\nabla$ and $\nabla '$. The condition that the residues are the same means that $\psi$ has neither pole nor zero
along each component $D_i$. Thus, $\psi$ is an isomorphism of bundles over $X$.

For surjectivity, suppose $(\rho , a_1,\ldots , a_k)$ is a point in $M_B(X,\log D)$. Choose a line bundle with logarithmic connection 
$(L,\nabla )$ inducing the monodromy representation $\rho$ on $U$. Let $a'_1,\ldots , a'_k$ be the residues 
of $\nabla '$ along the $D_i$. We have $a_i = a'_i - g_i$ with $g_i\in \zz $. Now $(L^g, \nabla ^g)$ 
maps to $(\rho , a_1,\ldots , a_k )$. 
\eop

We have the conjugation isomorphism $\varphi : U \cong \overline{X}-\overline{D}$. If $\rho$ is a local system
on $U$ then we obtain $\varphi _{\ast}(\rho )$ a local system on $\overline{X}-\overline{D}$, defined by
$$
\varphi _{\ast}(\rho ) (\eta ) := \rho (\varphi ^{-1}\eta ).
$$

The divisor $\overline{D}$ breaks up into components $\overline{D}_1 + \ldots + \overline{D}_k$. Let
$\gamma _{\overline{D}_i}\in \pi _1(\overline{U},\overline{x})^{\perp}$  denote the local monodromy operator around $\overline{D}_i$. Then we have
$$
\varphi ^{-1}(\gamma _{\overline{D}_i}) = \gamma _{D_i} ^{-1}
$$
because $\varphi $ reverses orientation. Thus,
$$
\varphi _{\ast}(\rho )^{\perp}(\gamma _{\overline{D}_i}) = \rho ^{\perp}(\gamma _{D_i}) ^{-1}.
$$
Given a logarithm $a_i$ of $\rho ^{\perp}(\gamma _{D_i})$, its negative $-a_i$ is a logarithm of 
$\varphi _{\ast}(\rho )^{\perp}(\gamma _{\overline{D}_i})$. 
Therefore define the isomorphism 
$$
\varphi : M_B(X,\log D) \stackrel{\cong}{\rightarrow}M_B(\overline{X}, \log \overline{D})
$$
by 
$$
\varphi (\rho ; a_1,\ldots , a_k) := (\varphi _{\ast}(\rho ); -a_1, \ldots , -a_k ).
$$

Using the Riemann-Hilbert correspondence of Lemma \ref{rhlift} we can do the Deligne glueing exactly as before to get a moduli space
$$
M_{DH}(X,\log D)\rightarrow \pp ^1.
$$

The gauge group of meromorphic gauge transformations along the divisors $\Gg = \zz ^k$ acts on $M_{DH}(X,\log D)$ in the following way.
It acts in the canonical way on the first chart $M_{Hod}(X,\log D)$. On the other hand, a point $(g_1,\ldots , g_k)$ acts
by the canonical action of $(-g_1, \ldots , -g_k)$ on the chart $M_{Hod}(\overline{X},\log \overline{D})$, because of the sign change
in the definition of $\varphi $. The global quotient is the Deligne glueing considered previously: 
\begin{equation}
\label{dhgauge}
M_{DH}(U) = M_{\rm DH}(X,\log D) /\Gg \rightarrow \pp ^1.
\end{equation}
Over $\Gm \subset \pp ^1$ this quotient is isomorphic to $M_B(U)\times \Gm$ so it has a reasonable structure; however near the fibers over 
$0$ and $\infty$ the quotient is analytically stacky, with $\Gg$ contributing to the stabilizer group. 

\subsection{Exact sequences}
\label{sec-exact}
Since we are treating the case $r=1$, our moduli spaces are really just abelian cohomology groups, for 
example\footnote{In this discussion, we are ignoring the stabilizer group $\Gm$ in the stack structure on the moduli spaces.}
$$
M_B(U) = H^1(U,\Gm ).
$$
This may also be interpreted as a Deligne cohomology group, see \cite{EsnaultViehweg} \cite{Gajer} for example; 
we leave to the reader to make the link between our Hodge filtration
and the Hodge filtration on Deligne cohomology.

The exponential exact sequence for $U=X-D$ is
$$
0\rightarrow H^1(U,\zz ^{\perp})\rightarrow H^1(U,\cc ) \rightarrow M_B(U)\rightarrow H^2(U,\zz ^{\perp})\rightarrow H^2(U,\cc ).
$$
The exact sequence for the gauge group action is
$$
0 \rightarrow \Gg = \zz ^k \rightarrow M_B(X,\log D) \rightarrow M_B(U)\rightarrow 1.
$$
Let $W_1H^1(U)= H^1(X, \zz )$ denote the weight $1$ piece of the weight filtration. 
We have an exact sequence
$$
0\rightarrow W_1H^1(U) = H^1(X, \zz ^{\perp}) \rightarrow H^1(U,\zz ^{\perp}) \rightarrow \Gg = \zz ^k = H^2(X,U, \zz ^{\perp}) \rightarrow 
$$
$$
\rightarrow H^2(X,\zz ^{\perp}) \rightarrow H^2(U,\zz ^{\perp}) \rightarrow H^3(X,U,\zz ^{\perp})\rightarrow \ldots 
$$
The exponential exact sequence lifts to an exact sequence for the logarithmic space 
$$
0\rightarrow W_1H^1(U,\zz ^{\perp})\rightarrow H^1(U,\cc ) \rightarrow M_B(X,\log D)\rightarrow H^2(X,\zz^{\perp} )\rightarrow H^2(X,\cc ).
$$
The image of the connecting map in the first exponential exact sequence, is the torsion subgroup of $H^2(U)$. A duality calculation relates 
$H^3(X,U,\zz^{\perp} )$ to the $H^1(D_i, \zz ^{\perp})$ so this is torsion-free. Hence, any torsion element in 
$H^2(U,\zz ^{\perp})$ comes from $H^2(X,\zz ^{\perp})$. This fits in with the fact that any element of $M_B(U)$ admits a
Deligne canonical extension to a line bundle with logarithmic connection on $X$. 

These exact sequences all fit together into a diagram
$$
\begin{array}{rclclclc}
&                &                      &                  &                & 0              &\!\rightarrow\!           & H^1(X ,\zz ^{\perp})   \\
&                &                      &                  &                & \downarrow     &                          & \downarrow      \\
&                &                      &                  &                & 0              &\!\rightarrow    \!       & H^1(U ,\zz ^{\perp})   \\
&                &                      &                  &                & \downarrow     &                          & \downarrow      \\
& 0              &\!     \rightarrow \! & 0                &\!\rightarrow\! & \Gg = \zz ^k   &\!\stackrel{=}{\rightarrow}\! & H^2(X, U,\zz ^{\perp})   \\
& \downarrow     &                      & \downarrow       &                & \downarrow     &                          & \downarrow      \\
0\!\rightarrow \! 
&  H^1(X,\zz ^{\perp})   &\!     \rightarrow \! & H^1(U,\cc ) &\!\rightarrow\! & M_B(X,\log D)  &\!\rightarrow       \!    & H^2(X,\zz ^{\perp})   \\
&  \downarrow     &                     & \downarrow       &                & \downarrow     &                          & \downarrow      \\
0\!\rightarrow \! 
&  H^1(U,\zz ^{\perp}) &\!         \rightarrow\!& H^1(U,\cc ) &\!\rightarrow\! & M_B(U)    &\!\rightarrow     \!      & H^2(U,\zz ^{\perp}) \\
& \downarrow     &                      & \downarrow       &                & \downarrow     &                          & \downarrow      \\
\Gg \stackrel{=}{\rightarrow}\!
& H^2(X,U,\zz ^{\perp})&\!      \rightarrow  \! & 0                &\!\rightarrow\! & 0              &\!\rightarrow      \!     & H^3(X,U,\zz ^{\perp})  \, .
\end{array} 
$$
There are also exact sequences for localization near the singular points. The main one is
\begin{equation}
\label{basicexact}
0\rightarrow M_B(X) \rightarrow  M_B(X, \log D) \stackrel{{\rm res}}{\rightarrow} \cc ^k \rightarrow H^2(X,\Gm ) .
\end{equation}
It fits with the exponential exact sequence to give a diagram
$$
\begin{array}{cccccccc}
             & 0              &            & 0             &            &                &            &  \\
             & \downarrow     &            & \downarrow    &            &                &            &                   \\
0\rightarrow & H^1(X,\zz ^{\perp})    &\rightarrow & H^1(X,\zz ^{\perp})   &\rightarrow & 0              &\rightarrow & H^2(X,\zz ^{\perp})    \\
             & \downarrow     &            & \downarrow    &            & \downarrow     &            & \downarrow        \\
0\rightarrow & H^1(X,\cc )    &\rightarrow & H^1(U,\cc ) &\rightarrow & \cc ^k         &\rightarrow & H^2(X,\cc )    \\
             & \downarrow     &            & \downarrow    &            & \downarrow     &            & \downarrow            \\
0\rightarrow & M_B(X)         &\rightarrow & M_B(X, \log D)&\rightarrow & \cc ^k         &\rightarrow & H^2(X, \Gm )        \\
             & \downarrow     &            & \downarrow    &            & \downarrow     &            & \downarrow     \\
0\rightarrow & H^2(X,\zz ^{\perp})    &\rightarrow & H^2(X,\zz ^{\perp})   &\rightarrow & 0              &\rightarrow &  H^3(X,\zz ^{\perp})\,  .        \\
\end{array} 
$$
Dividing by the gauge group gives the diagram
$$
\begin{array}{cccccccc}
             & 0              &            & 0             &            &  0             &              &    \\
             & \downarrow     &            & \downarrow    &            & \downarrow     &              &                \\
0\rightarrow & H^1(X,\zz ^{\perp})    &\rightarrow & H^1(U,\zz ^{\perp}) &\rightarrow & \zz ^k         &  \rightarrow & H^2(X,\zz ^{\perp})       \\
             & \downarrow     &            & \downarrow    &            & \downarrow     &              & \downarrow           \\
0\rightarrow & H^1(X,\cc )    &\rightarrow & H^1(U,\cc ) &\rightarrow & \cc ^k         &\rightarrow   & H^2(X,\cc )   \\
             & \downarrow     &            & \downarrow    &            & \downarrow     &              & \downarrow           \\
0\rightarrow & M_B(X)         &\rightarrow & M_B(U)&\rightarrow     & (\Gm ^{\perp} ) ^k&\rightarrow   & H^2(X, \Gm )        \\
             & \downarrow     &            & \downarrow    &            & \downarrow     &              & \downarrow    \\
0\rightarrow & H^2(X,\zz ^{\perp})    &\rightarrow & H^2(U,\zz ^{\perp})&\rightarrow & 0              & \rightarrow  & H^3(X,\zz ^{\perp}) \, .             \\
\end{array} 
$$
The connected component of $M_B(U)$ containing
the identity representation, is a quotient:
$$
M_B(U)^o = H^1(U,\cc )/ H^1(U,\zz ^{\perp}).
$$
This identification is valid in the analytic category. 
The logarithmic space is obtained by dividing out instead by $W_1H^1(X -D, \zz ^{\perp}) = H^1(X, \zz ^{\perp})$:
$$
M_B(X,\log D)^o  = H^1(U,\cc )/ W_1H^1(U,\zz ^{\perp}).
$$
Let $\Gg ^o := \ker \left( \zz ^k \rightarrow H^2(X,\zz ^{\perp}) \right) $. Then $\Gg ^o$ acts on 
$M_B(X,\log D)^o$ with quotient $M_B(U)^o$. The above diagrams show that this is compatible with the
quotient descriptions.

\subsection{Compatibility with Hodge filtration}
\label{sec-compatibility}
The above diagrams can be replaced with the corresponding diagrams of twistor spaces over $\pp ^1$. 
This raises the question of showing that the maps preserve the twistor structure, another way of saying
that they should be compatible with the Hodge filtrations. 

Recall that the Hodge filtration and its complex conjugate for $H^1(U, \cc )$ lead to the twistor bundle 
$\xi (H^1(U, \cc ),F,\overline{F})$ over $\pp ^1$, see \cite{icm}. 
We can again take the quotient by the action of $H^1(U,\zz )$ or $W_1H^1(U,\zz )$.  This gives an identification of
the Deligne-Hitchin twistor space, at least for the connected component of the identity representation.

\begin{theorem}
\label{mhs-ident}
Denote by a superscript $(\; )^o$ the connected component of the identity representation. 
We have identifications of analytic spaces over $\pp ^1$,
$$
M_{\rm DH}(X,\log D)^{o} \cong \xi (H^1(U, \cc ),F,\overline{F}) / W_1H^1(U,\zz ^{\perp})
$$
and
$$
M_{\rm DH}(X,\log D)^{o}/\Gg ^o  \cong \xi (H^1(U, \cc ),F,\overline{F}) / H^1(U,\zz ^{\perp}).
$$
Thus, the maps in the above big diagrams are compatible with the twistor structures. 
\end{theorem}
\begin{proof}
Use a cocycle description of $H^1(U,\cc )$. Suppose we are given an open analytic covering of $X$ by 
open sets $U_i$, and let $U_{ij}:= U_i\cap U_j$ etc. Recall Grothendieck's theorem
$$
H^1(U,\cc ) = {\mathbb H}^1\left( \Oo _X \rightarrow \Omega ^1_X (\log D) \rightarrow \Omega ^2_X(\log D)\rightarrow \ldots \right) .
$$
An element here is given by a pair $(\{ g_{ij} \} , \{ a_i\} )$ where 
$$
g_{ij}\in \Oo _X(U_{ij}), \;\;\; a_i \in \Omega ^1_X(\log D) (U_i)
$$
and these satisfy the cocycle condition $g_{ij}+ g_{jk} + g_{ki}=0$, the compatibility condition 
$d(g_{ij}) = a_i -a_j$, and $d(a_i) = 0$. 
The image of this pair in $M_{DR}(X,\log D)$ is $(L,\nabla )$ where $L$ is the line bundle whose transition functions are $e^{g_{ij}}$,
and $\nabla := d + a_i$  over $U_i$, with $d$ being the constant connection with respect to the trivialization $L|_{U_i}\cong \Oo _{U_i}$. 
This image is compatible with the exponential map $H^1(U,\cc )\rightarrow M_B(U)$ via the Riemann-Hilbert correspondence. 

The Hodge filtration or Rees-bundle $\xi (H^1(U,\cc ), F)\rightarrow \aaa ^1$ may also be described as the bundle of triples 
$(\lambda , \{ g_{ij}\} , \{ a_i\} )$ subject to the conditions, analogues of the notion of $\lambda$-connection:
$$
g_{ij}+ g_{jk} + g_{ki}=0, \;\;\;
\lambda d(g_{ij}) = a_i -a_j, \;\;\; \lambda d(a_i) = 0.
$$
Map this triple to $(\lambda , L, \nabla )$ where $L$ is again given by transition functions $e^{g_{ij}}$, and 
$\nabla := \lambda d + a_i$  over $U_i$. This gives a map
$$
\xi (H^1(U,\cc ), F) \rightarrow M_{\rm Hod}(X,\log D).
$$
It is compatible with the action of $\Gm$, and is the same as the previous map in the fiber over $\lambda = 1$,
so it is compatible with the exponential map on Betti cohomology.

Note that the complex conjugate of the Hodge filtration on $H^1(U,\cc )$ is the same as the pullback by $\varphi : U^{\rm top}\cong \overline{U}^{\rm top}$,
of the Hodge filtration on $H^1(\overline{U},\cc )$. Indeed, $\varphi$ is antiholomorphic so
the pullback by $\varphi$ of a cohomology class containing at least a certain number of $dz_i$,
is a cohomology class containing at least that many $d\overline{z}_i$. 

Using all of these things, our map glues together with the corresponding map
for $\overline{U}= \overline{X}-\overline{D}$ to give a map of twistor spaces over $\pp ^1$,
$$
\xi (H^1(U,\cc ), F,\overline{F}) \rightarrow M_{\rm DH}(X,\log D).
$$
This is the required compatibility. 

From the cocycle description, we get that the map is surjective to the connected component
$M_{\rm DH}(X,\log D)^{o} $, 
even in the fibers over $\lambda = 0,\infty $. Using smoothness of both sides over $\pp ^1$ and a dimension count, 
we see that the kernel is discrete and flat over $\pp ^1$. In the general fiber it is $W_1H^1(U,\zz )$. The closure of the graph of this subgroup is again
a subgroup of the form $W_1H^1(U,\zz )\times \pp ^1\subset \xi (H^1(U,\cc ), F,\overline{F})$. 
Hence the isomorphism 
$$
M_{\rm DH}(X,\log D)^{o} \cong \xi (H^1(U, \cc ),F,\overline{F}) / W_1H^1(U,\zz ).
$$
The other one is obtained by dividing out by the gauge group $\Gg ^o$. 
\end{proof}

{\em Problem:} Find a similar description for the twistor spaces of 
other connected components of $M_B(U)$ corresponding to torsion elements in $H^2(U,\zz )$. This should be doable using the fact
that the points of finite order in $M_B(U)$ occur in every connected component, as can be seen from the analogue of the
exponential exact sequence 
$$
0\rightarrow H^1(U,\zz ) \rightarrow H^1(X_D, \qq ) \rightarrow H^1(U, \mu _{\infty}) \rightarrow H^2(U,\zz ) \rightarrow H^2(U,\qq ).
$$

\subsection{Preferred sections}
\label{sec-preferred}
We now describe how a tame harmonic bundle of rank one on $U=X-D$ gives rise to a section of the fibration \eqref{dhgauge}.  
In this discussion, we use the notion of parabolic structure and in particular Mochizuki's notion of KMS-spectrum 
\cite{Mochizuki}. See also Budur \cite{Budur} for a discussion of the rank one case.
In Theorem \ref{pref-id}, the space of harmonic bundles will be identified with the space of $\sigma$-invariant sections of 
$M_{\rm DH}(U)$. The latter doesn't refer to the notion of parabolic structure, but the identification map between them does. 

Fix a K\"ahler  metric $\omega$ on $X$, which restricts to a K\"'ahler metric on $U$. 
Recall that a tame harmonic
bundle over $U$ is a vector bundle $E$ together with operators $D'$ and $D''$ and a metric $h$, with respect to which these
operators satisfy certain equations \cite{Hitchin} \cite{Corlette} \cite{hbls}. Our preferred section will not depend on changes of $h$ by multiplying by a positive constant. Decompose 
\begin{equation}
\label{decomp}
D'' = \delbar + \theta , \;\;\; D' = \partial + \overline{\theta} .
\end{equation}
Use the notation $\Ee = (E,D',D'',h)$ for our harmonic bundle. 

Fix $\lambda \in \pp ^1$ and for now we suppose it is in the first standard chart $\aaa ^1$ so we think of $\lambda \in \cc$. Then
we get a holomorphic structure $\delbar + \lambda \overline{\theta}$ on the bundle $E$, and a  $\lambda$-connection $\lambda \partial + \theta$. 
Denote the holomorphic bundle with this holomorphic structure by $\Ee ^{\lambda}$ and the $\lambda$-connection by $\nabla ^{\lambda}$
By \cite{Mochizuki}, for any vector $a= (a_1,\ldots , a_k)$ of real numbers, we get an extension
of $\Ee ^{\lambda}$ to a  holomorphic bundle denoted $E^{\lambda}_a$ on $X$, and $\nabla ^{\lambda}$ extends to a logarithmic
$\lambda$-connection on  $E^{\lambda}_a$.  

If we pick $\lambda _0$ and any $i=1,\ldots , k$, then there is a set of critical values of $a_i$ called the {\em KMS-spectrum} \cite{Mochizuki}. 
For $a= (a_1,\ldots , a_k)$ with $a_i$ not in the KMS-spectrum at $\lambda _0$ and $D_i$,
there is a neighborhood $\lambda _0\in L\subset \pp ^1$ such that for $\lambda \in L$, the bundles with logarithmic connection
$(\Ee ^{\lambda}_a, \nabla ^{\lambda})$ vary holomorphically in $\lambda$. For each divisor component and fixed $\lambda _0$,
the KMS-spectrum is a $\zz$-translation orbit in $\rr$, that is it consists of everything of the form $a_i+u_i$ for $u_i\in \zz$. 
This is special to the rank one case, where there is only one KMS spectrum element in $\rr /\zz$.

The {\em KMS-critical locus} at $\lambda _0$ is the set of
all $a$ such that some $a_i$ is in the KMS-spectrum for $\lambda _0$ and $D_i$. This locus is a union of translates of the $k$ coordinate
hyperplanes. The translates included are all of the form $(a_1 + u_1,\ldots , a_k + u_k)$ where $a_i$ are some elements of the KMS-spectrum, 
and $u_i$ are any integers. The {\em KMS-chambers} are the connected components of the complement of the KMS-spectrum. Note that $\zz ^k$ acts
simply transitively on the set of KMS-chambers for any $\lambda_0$. Furthermore, the set of KMS-chambers varies continuously with $\lambda _0$:
a point which is well in the middle of a chamber for $\lambda _1$, will remain in a uniquely determined chamber for $\lambda _1$ nearby,
or to put it another way the KMS-critical locus varies continuously as a function of $\lambda$.  

In particular, if for any $\lambda _0$ we choose a particular KMS-chamber, then by following this around it determines a KMS-chamber for all other
$\lambda \in \aaa ^1$.

For different values of $a$ in the same KMS-chamber, the bundles with logarithmic connection $(\Ee ^{\lambda}_a, \nabla ^{\lambda})$ 
are all canonically isomorphic. Hence, the choice of a KMS-chamber for a given $\lambda _0$ determines a choice of KMS-chamber for
all $\lambda \in \aaa ^1$; let $a(\lambda )$ denote a function taking values in this chamber for each $\lambda$. 
We thus get the collection of bundles with logarithmic connection depending on $\lambda$,
$$
\lambda \mapsto (\Ee ^{\lambda}_{a(\lambda )}, \nabla ^{\lambda}).
$$
This is our preferred section of $M_{DH}(X,\log D)$ over $\aaa ^1$. If we choose a different chamber to begin with, then the section is modified by the
corresponding element of the local meromorphic gauge group $\Gg = \zz ^k$. The projection to the quotient gives a uniquely defined section of the
fibration \eqref{dhgauge}, at least over $\aaa ^1$. 

This construction patches together with the corresponding construction on the other chart $\aaa ^1$ at $\infty$. 
See \cite[Chapter 11]{Mochizuki}.

The construction we have described here is an isomorphism between harmonic bundles and $\sigma$-invariant sections of the fibration 
$M_{\rm DH}(X,\log  D)/\Gg  \rightarrow \pp ^1$. 

\begin{theorem}
\label{pref-id}
Let $M_{\rm har}(U)$ denote the group of tame harmonic line bundles on $U$. The map described above goes from here to the space of
$\sigma$-invariant sections of $M_{\rm DH}(X,\log  D)$ modulo the gauge group action:
$$
\Pp : M_{\rm har}(U) \rightarrow \Gamma (\pp ^1 , M_{\rm DH}(X,\log D)) ^{\sigma} /\Gg .
$$
This map is an isomorphism. 
\end{theorem}

The map is given by the discussion above. The proof that it is an isomorphism, which requires techniques from the next subsections, 
will be given in Corollary \ref{pref-id-inj} and \S \ref{sub-proof} below. 

This theorem, which is only in the rank one case, nevertheless suggests that in the correspondence between harmonic bundles and 
pure twistor $\Dd$-modules of \cite{Mochizuki2} and \cite{Sabbah} the parabolic weight should come out of the structure of 
twistor $\Dd$-module, without having to impose an additional parabolic structure on the $\Dd$-module side. It isn't clear to me
to what extent this statement may already be contained in \cite{Mochizuki2} and \cite{Sabbah}. 

\subsection{Residues and parabolic structures}
\label{sub-parabolic}

We now get to one of the main observations in this article: that the three dimensional space of $\sigma$-invariant sections of $T(1,\log )$ 
encodes the data of residues and parabolic weights for a harmonic bundle. 

Fix a divisor component $D_i$, and a point $p\in \aaa ^1$. The fiber $T(1,\log )_p$ is identified with $\cc$ by the frame 
$\frac{\partial}{\partial \lambda}$. Hence the residue map can be composed with this identification to give
$$
{\rm res}_{D_i,p}:M_{\rm Hod}(X,\log D)_p \rightarrow \cc \cong T(1,\log )_p.
$$
It sends a logarithmic $p$-connection $(E,\nabla )$ to ${\rm res}(\nabla ; D_i) \frac{\partial}{\partial \lambda}(p)$.

The glueing function for residues of logarithmic $\lambda$-connections is $-\lambda ^2$, the same as for $T(1,\log )= T\pp ^1$.
Therefore, this map glues with the same map on the chart $M_{\rm Hod}(\overline{X},\log \overline{D})$ to give a bundle map over $\pp ^1$,
$$
{\rm res}^{\rm DH}_{D_i}: M_{\rm DH}(X,\log D)\rightarrow T(1,\log ).
$$

\begin{lemma}
\label{residuesigma}
The residue map ${\rm res}^{\rm DH}_{D_i}$ is compatible with the antipodal involutions on $M_{\rm DH}(X,\log D)$ and 
$T(1,\log )$, so it gives a map on $\sigma$-invariant sections 
$$
\Gamma (\pp ^1, M_{\rm DH}(X,\log D)) ^{\sigma} \rightarrow \Gamma (\pp ^1, T(1,\log )) ^{\sigma}. 
$$
\end{lemma}
\begin{proof}
The calculation for $X$ near $D$ is the same as that of \S \ref{sub-antipodal}. Comparing with
the calculation of \S \ref{sub-tate-antipodal}, we see that the residue is compatible with $\sigma$
and it induces a map on $\sigma$-invariant sections. 
\end{proof}

Next, consider the projection ${\rm pr}_i: \Gg \rightarrow \zz (1,\log )$ which sends $(g_1,\ldots , g_k)$ to $g_i$. 

\begin{lemma}
\label{residuegauge}
The residue map ${\rm res}^{\rm DH}_{D_i}$ is compatible with the action of the local meromorphic gauge group $\Gg$
via the projection ${\rm pr}_i$ composed with the morphism $\zz (1,\log )\rightarrow T(1,\log )$, so it gives a map on quotients
$$
\Gamma (\pp ^1, M_{\rm DH}(X,\log D)) ^{\sigma} /\Gg \rightarrow \Gamma (\pp ^1, T(1,\log )) ^{\sigma}/\zz (1,\log ). 
$$
\end{lemma}
\begin{proof}
At each point $p$, the action of the gauge group is compatible by equation \eqref{residueformula}
with the map ${\rm pr}_i$ via the standard morphism $\zz (1,\log ) \rightarrow T(1,\log )_p$
which sends the generator to $-p\frac{\partial}{\partial\lambda}$. 
This gives the compatibility on global sections.
\end{proof}

For any $p\in\aaa ^1\subset \pp ^1$, we can compose the map of Lemma \ref{residuesigma} with the isomorphism
$(\varpi _p,{\rm res}_p)$ of Proposition \ref{prop-parabolicweight}, comprising the parabolic weight function 
$\varpi _p$ and the residue or evaluation at $p$. This gives a map
\begin{equation}
\label{globalsigma}
(\varpi _p, {\rm res}_p)_{D_i} : \Gamma (\pp ^1, M_{\rm DH}(X,\log D)) ^{\sigma} \rightarrow \rr \times \cc .
\end{equation}
Dividing by the action of the local meromorphic gauge group corresponds to dividing by the action of $\zz$ on $\rr\times \cc$ generated
by $(\varpi _p, {\rm res}_p)(\psi (1,0))= (1,-p)$. We get a quotient map
\begin{equation}
\label{globalsigmagauge}
(\varpi _p, {\rm res}_p)^{\Gg}_{D_i} : \Gamma (\pp ^1, M_{\rm DH}(X,\log D)) ^{\sigma}/\Gg  \rightarrow \frac{\rr \times \cc}{(1,-p)\cdot \zz} .
\end{equation}

Compose  with the preferred-sections  map $\Pp$ of  Theorem \ref{pref-id}. Our main observation is that this 
encodes the parabolic weight and residue of a harmonic bundle. These were defined for the case of curves, at $\lambda = 0$ and 
$\lambda = 1$, in \cite{hbnc}. They were defined in higher dimensions and for all $\lambda$ in \cite{Mochizuki}. 

Given a parabolic bundle $F=\{ F_b\}$ filtered by bundles indexed in the increasing sense by $b\in \rr ^k$, suppose we have chosen $E$ as one of these bundles. Define the
{\em parabolic weight} to be the element $b=(b_1,\ldots , b_k)$ with $b_i$ as small as possible so that $E=F_b$. Given a harmonic bundle 
$\Ee = (E,D',D'',h)\in M_{\rm Har}(U)$, we obtain for any $\lambda$ a parabolic logarithmic $\lambda$-connection $\Ee ^{\lambda}$ by 
\cite{Mochizuki}. Its underlying parabolic
bundle has a parabolic weight as defined at the start of this paragraph, and the parabolic $\lambda$-connection on $\Ee ^{\lambda}$ has a residue along
each $D_i$. The parabolic weight of $\Ee ^{\lambda}$ is determined by the rate of growth of the harmonic metric: if $u$ is a unit section near a point of $D_i$,
and if $D_i$ is cut out by the equation $z=0$, then $|u|_h\sim |z|^{-b_i}$ where $b_i$ is the parabolic weight along $D_i$. 

\begin{theorem}
\label{harmoniccompose}
Suppose $D_i$ is a divisor component and $p\in \aaa ^1\subset \pp ^1$. Suppose $\Ee = (E,D',D'',h)\in M_{\rm Har}(U)$ is a rank one harmonic bundle on $U$. 
Then 
$$
(\varpi _p,{\rm res}_p)^{\Gg}_{D_i} (\Pp (\Ee )) \in \frac{\rr \times \cc}{(1,-p)\cdot \zz} 
$$
is the parabolic weight and residue of the parabolic logarithmic $\lambda$-connection $\Ee ^{\lambda}$ at $\lambda = p$.
\end{theorem}
\begin{proof}
Fix an extension of the logarithmic Higgs bundle $(\Ee ^0,\nabla ^0)$ to a line bundle over $X$. 
It then has a harmonic metric $h$. Let $a'$ be the parabolic weight along $D_i$. 
Let $\alpha '$ be the residue of the Higgs field along $D_i$. Mochizuki defines functions ${\mathfrak p}(\lambda ,a,\alpha )$ and ${\mathfrak e}(\lambda ,a,\alpha )$
in \cite[\S 2.1.7]{Mochizuki}, and in Corollary 7.71, \cite[\S 7.3.3]{Mochizuki}
he points out that the rule obeyed by the KMS-spectrum of a harmonic bundle is given by the transformation 
$({\mathfrak p},{\mathfrak e})$. 
In the rank one case, the KMS-spectrum has only one
element. Hence, the transformation rule \cite[Cor. 7.71]{Mochizuki} means that the parabolic weight and residue of $\Ee ^{\lambda }$ are respectively
$$
{\mathfrak p}(\lambda ,a',\alpha ')\;\;\; \mbox{and} \;\;\; {\mathfrak e}(\lambda ,a',\alpha ').
$$

By inspection, the functions ${\mathfrak p},{\mathfrak e}$ of \cite[\S 2.1.7]{Mochizuki} are the same 
as the parabolic weight functions and residue functions occuring in Proposition \ref{prop-parabolicweight}:
\begin{equation}
\label{equmochi}
(\varpi _p,{\rm res}_p)(\psi (a,\alpha )) = ({\mathfrak p}(p,a,\alpha ), {\mathfrak e}(p,a,\alpha )). 
\end{equation}

Recall that  $\lambda \mapsto (\Ee ^{\lambda} ,\nabla ^{\lambda})$ is exactly our preferred section $\Pp (\Ee )$ (lifted over the gauge group action). 
Let $(a,\alpha )$ denote the Higgs coordinates for the residue section, so
$$
{\rm res}_{D_i}(\Pp (\Ee ))=\psi (a,\alpha ) \in \Gamma (\pp ^1,T(1,\log ))^{\sigma} \cong \rr \times \cc.
$$ 
The value of this section at the point $\lambda$, which is 
the residue of the logarithmic $\lambda$-connection $\nabla ^{\lambda}$, is given by the residue function 
${\rm res}_{\lambda}(\psi (a,\alpha ))$ calculated in \S \ref{sub-reseval} above. 
We conclude that for all $\lambda \in \aaa ^1$,
$$
{\rm res}_{\lambda}(\psi (a,\alpha )) = {\mathfrak e}(\lambda ,a',\alpha ').
$$
The identity \eqref{equmochi} between the functions ${\rm res}_{\lambda}(\psi (a,\alpha ))$ and
${\mathfrak e}(\lambda ,a,\alpha )$, writing them out per \S \ref{sub-reseval}, means that 
$$
\alpha - a\lambda  -\overline{\alpha}\lambda ^2 = \alpha ' - a'\lambda  -\overline{\alpha}'\lambda ^2
$$
for all $\lambda \in \aaa ^1$. It follows that $a=a'$ and $\alpha = \alpha '$. This proves the statement of the theorem at $p=\lambda = 0$. 

At a general value of $\lambda = p$, the parabolic weight and residue of the harmonic bundle are given as 
${\mathfrak p}(\lambda ,a,\alpha )$ and ${\mathfrak e}(\lambda ,a,\alpha )$ respectively, by
Corollary 7.71, \cite[\S 7.3.3]{Mochizuki}. The identity \eqref{equmochi} shows that these are the same
as $(\varpi _p,{\rm res}_p)_{D_i} (\Pp (\Ee ))$. Modulo the action of the gauge group (which absorbs our initial choice of extension of the bundle), 
this gives the statement of the theorem. 
\end{proof}

We now have enough to do half of the isomorphism in Theorem \ref{pref-id}.

\begin{corollary}
\label{pref-id-inj}
The preferred-sections morphism $\Pp$ in Theorem \ref{pref-id} is injective. 
\end{corollary}
\begin{proof}
Suppose $\Ee$ and $\Ff$ are rank one harmonic bundles, such that $\Pp (\Ee ) \cong \Pp (\Ff )$. The local parabolic weight and residue data of the harmonic
bundles coincide,
because these functions factor through $\Pp$ by Theorem \ref{harmoniccompose}. The line bundles with connection ($\lambda = 1$) associated to $\Ee$ and
$\Ff$ correspond to filtered local systems of rank $1$ \cite{hbnc} \cite{Mochizuki} \cite{GukovWitten}. The filtration weight of the filtered local system
is obtained from the parabolic weight and residues of the line bundles with connection, see the third column of the table in \cite[p. 720]{hbnc}. 
Therefore, the filtration weights of the filtered local systems associated to $\Ee$ and $\Ff$ are the same. The fact that
$\Pp (\Ee ) \cong \Pp (\Ff )$ at $\lambda = 1$ restricted over $\lambda = 1$ means 
that the associated logarithmic connections are the same up to local meromorphic gauge transformation, hence the associated monodromy representations are the same.
Now, in rank one a filtered local system is determined uniquely by its monodromy representation and its filtration weight. 
A filtered local system corresponds to a unique  harmonic bundle. Therefore $\Ee \cong \Ff$. 
\end{proof}

\subsection{Comparison with \cite{hbnc}}
\label{comp-hbnc}

It is interesting to comment on the particular cases $p=0$ and $p=1$. 
The transformation for going from $p=0$ to $p=1$ gives back the 
transformation between the first two columns of the table on p. 720 of \cite{hbnc}, 
which has remained mysterious to me up until now.  
The parabolic weights and residues for the
different points $\lambda = p$, are different coordinate systems on the same three dimensional space $\Gamma (\pp ^1, T(1,\log )) ^{\sigma}$.
Going between two different values of $p$ gives a change of coordinates. In \cite{hbnc} only the values $\lambda = 0$ (Higgs bundles) and $\lambda = 1$ 
(logarithmic connections)
were considered. However there are some changes of notation: we have adopted Mochizuki's coordinates \cite[\S 6.1.1]{Mochizuki} at the Higgs point $(a,\alpha )$ for the reader's convenience. 
We have also adopted the standard convention that the parabolic structure is indexed by an increasing filtration. 

In \cite{hbnc}, the parabolic structure was given by a decreasing filtration. 
So, here $a\in \rr$ is the parabolic weight of the Higgs bundle in the increasing sense, which corresponds to $-\alpha$ in the notation of \cite{hbnc}. 
In \cite{hbnc} the sheaf $E_{\alpha}$ corresponded to sections whose growth was bounded by $|z|^{\alpha}$ whereas here $E_a$ corresponds to 
sections whose growth is bounded by $|z|^{-a}$. 

And here, $\alpha$ is
the residue of the Higgs field, which was denoted by $b+ci$ on p. 720 of \cite{hbnc}. These coordinates coincide with our parabolic weight and residue
at $p=0$. The parabolic weight and residue at $p=1$ are given by the formulae \eqref{for-paraweight} \eqref{for-reseval}
$$
\varpi _1(a,\alpha ) = a+\alpha +\overline{\alpha}, \;\;\; {\rm res}_1 (a,\alpha ) = \alpha - \overline{\alpha}-a.
$$
In terms of the notation $(\alpha , b,c)$ of \cite{hbnc}---where $\alpha$ has a different meaning from the rest of the present paper, and where $i\in \cc$ is 
chosen---we get
$$
\varpi _1 = -\alpha + 2 b, \;\;\; {\rm res}_1 = \alpha + 2 i c.
$$
These are the values in the second column of the table on page 720 of \cite{hbnc}, taking into account that the ``jump'' there is $-\varpi _1$. 

The conclusion is that the three-dimensional space, and the coordinate
transformation in the table of \cite{hbnc}, come from the fact that the twistor bundle of residues is $T(1, \log )\cong \Oo _{\pp ^1}(2)$,
in other words the local monodromy around singular divisors is in a weight-two twistor bundle.

\subsection{The weight filtration}
\label{sec-weight}

Because we look at line bundles, the moduli spaces form groups under tensor product. For $M_{\rm DH}$ we get
a group structure relative to $\pp ^1$. Define the {\em weight filtration}:
$$
W_1M_{\rm DH}(X,\log D) := M_{\rm DH}(X)^o, \;\;\; W_2M_{\rm DH}(X,\log D) := M_{\rm DH}(X,\log D).
$$
In our arguments below, it has seemed most natural to use only the connected component of the identity representation in the
weight $1$ piece.  

Define the second graded piece as the quotient using the group structure
$$
Gr^W_2(M_{\rm DH}(X,\log D)) := W_2/W_1 = \frac{M_{\rm DH}(X,\log D)}{M_{\rm DH}(X)} .
$$
Using only the connected component $M_{\rm DH}(X)^o$ for $W_1$, leads to a nontrivial finite group as $Gr^W_2$ even in the 
compact case $D=\emptyset$.  

There is a version modulo the gauge group:
$$
Gr^W_2(M_{\rm DH}(U)) = \frac{M_{\rm DH}(U)}{M_{\rm DH}(X)^o} = \frac{Gr^W_2(M_{\rm DH}(X,\log D))}{\Gg} .
$$
In the second equality the quotient is taken in a stacky sense over $\lambda = 0,\infty$. 
The inclusion $M_{\rm DH}(X)^o\subset M_{\rm DH}(U)$ is strict, injective on each fiber over $\lambda \in \pp ^1$. 
There is also an induced weight filtration on the connected component given by 
$$
W_1M_{\rm DH}(X,\log D)^o := M_{\rm DH}(X)^o. 
$$

\begin{lemma}
\label{weight-compatible}
The exponential map giving an isomorphism in Theorem \ref{mhs-ident} is strictly compatible with the weight filtrations for connected components on both sides, in
other words it sends the usual weight filtration on $H^1(U,\cc )$ to the weight filtration on $M_{\rm DH}(X,\log D)^o$. 
\end{lemma}
\begin{proof}
The exact diagrams in \S \ref{sec-exact} extend to exact diagrams of bundles over $\pp ^1$, with the Tate twistor structure $T(1,\log )$
inserted in place of $\cc$ at appropriate  places. The exponential map is one of the middle vertical maps in the exact squares. The weight foliation
is the kernel foliation of the map $R$, and the weight filtration on abelian cohomology is the kernel of the corresponding
map $H^1(U,\cc )\rightarrow \cc ^k$. Exactness then implies that the exponential isomorphism is compatible with weight filtrations. 
\end{proof}

The goal of this subsection is to identify $Gr^W_2M_{\rm DH}(X,\log D)$ and show that it has ``weight $2$''.

Write $NS(X)$ for the Neron-Severi group of divisors modulo algebraic equivalence, which is contained in
$H^2(X,\zz )$.
Let $NS(X,D)\subset NS(X)$ be the subgroup generated by divisor components $D_i$ of $D$. Let
$NS(X,D)^{\rm sat}$ be the saturation of this subgroup, in other words the subgroup of all elements $A\in NS(X)$ such
that some multiple $mA$ is in $NS(X,D)$. This may be seen as the kernel in the sequence
$$
0\rightarrow NS(X,D)^{\rm sat} \rightarrow H^2(X,\zz ) \rightarrow \frac{H^{1,1}(X,\cc)}{\cc \cdot [D_1] + \cc \cdot [D_k]} 
$$
and it includes the subgroup of torsion  $NS(X)^{\rm tors}$. For example if $D=\emptyset$ then 
$NS(X,D)^{\rm sat}  = NS(X)^{\rm tors}$. 
We close this paragraph by noting that 
$NS(X,D)^{\rm sat}$ is the preimage of $NS(U)^{\rm tors}$ under the restriction map $NS(X)\rightarrow NS(U)$, and
\begin{equation}
\label{nsutors}
\frac{NS(X,D)^{\rm sat}}{NS(X,D)} = NS(U)^{\rm tors}.
\end{equation}

Putting together the residue maps at divisor components $D_1,\ldots , D_k$, we get a map 
$$
R: M_{\rm DH}(X,\log D)\rightarrow T(1,\log ) ^k.
$$
The condition \eqref{c1cond} on $c_1(L)$ for a point $(L,\nabla )\in M_{\rm DH}(X,\log D)$ is equivalent to saying that $c_1(L)\in NS(X,D)^{\rm sat}$.
These give a map
$$
(c_1,R): M_{\rm DH}(X,\log D)\rightarrow NS(X,D)^{\rm sat}\times T(1,\log )^k.
$$

Consider the map defined using the divisor components
$$
\Sigma :\cc ^k \rightarrow H^{1,1}(X,\cc ), \;\;\; (a_1,\ldots , a_k)\mapsto \sum a_i[D_i].
$$
The cohomology $H^2(X,\cc )$ has a pure weight two Hodge structure, whose twistor bundle is semistable of slope $2$. 
The twistor bundle $\xi (H^2(X,\cc ), F,\overline{F})$ 
has a natural subbundle corresponding to $H^{1,1}$, and since that space is pure of Hodge type $(1,1)$ we have a
natural isomorphism 
$$
\xi (H^{1,1}(X,\cc ), F,\overline{F}) \cong H^{1,1}(X,\cc )\otimes  T(1,\log ).
$$
With respect to these constructions, the map $\Sigma$ extends over $\pp ^1$ to give a map
$$
\Sigma _{\rm DH}: T(1,\log )^k \rightarrow H^{1,1}(X,\cc )\otimes T(1,\log ).
$$
There is a natural morphism of groups over $\pp ^1$, 
$$
\Lambda _{\rm DH}: NS(X)\times \pp ^1\rightarrow H^{1,1}(X,\cc )\otimes T(1,\log )\cong
\xi (H^{1,1}X,\cc ), F,\overline{F})
$$
which sends an element of $NS(X)$ to a section which has a simple pole at $0$ and another simple pole at $\infty$. 
In terms of the usual trivialization of $T(1,\log )$ over $\aaa ^1$, this corresponds to multiplying by $\lambda$ the 
usual map from $NS(X)$ to $H^{1,1}(X,\cc )$. 
Restrict it to the subgroup $NS(X,D)^{\rm sat}$. Adding to $\Sigma _{\rm DH}$ gives a morphism
$$
\Lambda _{\rm DH} + \Sigma _{\rm DH}: NS(X,D)^{\rm sat}\times T(1,\log )^k \rightarrow H^{1,1}(X,\cc )\otimes T(1,\log ).
$$
The Chern classes and residues of logarithmic $\lambda$-connections on line bundles $(L,\nabla )$ satisfy the condition \eqref{c1compatible},
which in terms of the present notation says that 
the composed map $(\Lambda _{\rm DH} + \Sigma _{\rm DH})\circ (c_1,R)$ is zero. 

The following exact sequence is an analogue of the basic exact sequence \eqref{basicexact} and following diagram in \S \ref{sec-exact}.

\begin{proposition}
\label{weight-exact}
The weight filtrations for $M_{\rm DH}(X,\log D)$ and $M_{\rm DH}(U)$ with group structure given by tensor product, 
fit into a strict exact sequence of analytic groups over $\pp ^1$
\begin{equation}
\label{exseq}
1\rightarrow M_{\rm DH}(X)^o\rightarrow M_{\rm DH}(X,\log D) \stackrel{(c_1,R)}{\longrightarrow} 
NS(X,D)^{\rm sat}\times T(1,\log )^k \cdots 
\end{equation}
$$
\cdots \stackrel{\Lambda _{\rm DH} + \Sigma _{\rm DH}}{\longrightarrow} 
H^{1,1}(X,\cc )\otimes T(1,\log ).
$$
\end{proposition}
\begin{proof}
It suffices to prove this in the fiber over a fixed $\lambda$, and we may assume $\lambda \in \aaa ^1$. 
Injectivity on the left is easy and was mentionned previously: given $(L,\nabla )$ and $(L',\nabla ')$ on $X$,
an isomorphism of logarithmic $\lambda$-connections on $(X,D)$ between them is also an isomorphism of $\lambda$-connections on $X$. 

For exactness at $M_{\rm DH}(X,\log D)$, suppose $(L,\nabla )$ is a logarithmic $\lambda$-connection such that $c_1(L)=0$ and 
$R(\nabla )=0$, that is ${\rm res}_{D_i}(\nabla )= 0$ for each component $D_i$. Then $\nabla$ is a connection over $X$ and 
even if $\lambda = 0$, the condition $c_1=0$ insures inclusion in $M_{\rm DH}(X)$. The fact that $c_1=0$ in the Neron-Severi
group means that $L$ is algebraically equivalent to $0$. From the structure of the Picard group this implies that $L$ is
in the connected component of the trivial bundle and $M_{\rm DH}(X)_{\lambda}\rightarrow Pic(X)$ is smooth, so $(L,\nabla )$
is in the connected component $M_{\rm DH}(X)^o_{\lambda}$.

We prove exactness at $T(1,\log )^k \times NS(X,D)^{\rm sat}$. 
Use the standard frame for $T(1,\log )$. 
A point in $\ker (\Lambda _{\rm DH} + \Sigma _{\rm DH})$ is $\zeta \in NS(X,D)^{\rm sat}$ together with a $k$-uple $(a_1,\ldots , a_k)\in \cc ^k$ 
such that 
$$
\lambda \zeta + \sum a_i[D_i] = 0 \;\;\; \mbox{in}\;\; H^{1,1}(X,\cc ).
$$
Choose a line bundle $L$ such that 
$c_1(L) = \zeta $. 
The elements of $NS(X,D)^{\rm sat}$ restrict to torsion elements on $U$ by \eqref{nsutors}. 
Line bundles whose Chern class are torsion, have flat regular singular connections, which in the rank $1$ case are automatically 
logarithmic. Thus we can choose an initial $\lambda$-connection $\nabla '$ on $L$
logarithmic with respect to $(X,D)$. Let $a'_i$ denote the residues of $\nabla '$ along $D_i$. Then 
$$
\sum (a_i-a'_i)[D_i] = 0 \;\; \;\mbox{in}\;\; H^{1,1}(X,\cc ).
$$
Hence there is a logarithmic one-form $\beta$ on $(X,D)$ having residues $a_i-a'_i$ along $D_i$. 
Now $\nabla = \nabla ' + \beta $ is a logarithmic $\lambda$-connection 
with $(c_1,R)(L,\nabla ) = (\zeta , (a_1,\ldots , a_k))$.
\end{proof}

\begin{corollary}
\label{grw2-ident}
The exact sequence \eqref{exseq} identifies the graded piece of the weight filtration as
$$
Gr^W_2 M_{\rm DH}(X,\log D)= 
$$
$$
\ker \left( NS(X,D)^{\rm sat} \times T(1, \log ) ^k  \stackrel{\Lambda _{\rm DH} + \Sigma _{\rm DH} }{\longrightarrow}  
H^{1,1}(X,\cc )\otimes T(1,\log ) \right) .
$$
\end{corollary}

We can now describe the weight two phenomenon in the title of the paper:

\begin{lemma}
\label{weight2}
There is $b$ such that
$$
\ker \left( T(1,\log ) ^k \rightarrow H^{1,1}(X,\cc )\otimes T(1,\log )  \right) \cong T(1,\log )^b.
$$
There is an exact sequence
$$
0 \rightarrow T(1,\log )^b \rightarrow Gr^W_2 M_{\rm DH}(X,\log D) \rightarrow NS(X,D)^{\rm sat} \rightarrow 0.
$$
On the connected component of the identity representation
$$
T(1,\log )^b \cong Gr^W_2 M_{\rm DH}(X,\log D)^o  .
$$
Modulo the gauge group we have
$$
0 \rightarrow \Gm (1)^b \rightarrow Gr^W_2 M_{\rm DH}(U) \rightarrow NS(U)^{\rm tors} \rightarrow 0.
$$
\end{lemma}
\begin{proof}
A map between pure twistor structures of weight $2$ has a kernel which is again a pure twistor structure of weight $2$. 
Hence, there is $b$ as in the first claim. The first exact sequence comes from Proposition \ref{weight-exact} and
the fact that every element of $NS(X,D)^{\rm sat}$ goes into $H^{1,1}(X,\cc )$ to something which comes from $\cc ^k$. 
For the second exact sequence, note that 
$$
Gr^W_2 \left( M_{\rm DH}(X,\log D)^o\right) = \left( Gr^W_2 M_{\rm DH}(X,\log D)^o\right)
$$
because $W_1M_{\rm DH}(X,\log D)^o = M_{\rm DH}(X,\log D)$. 

For the last exact sequence, divide out by the gauge group $\Gg = \zz ^k$, which means dividing the first exact sequence by the exact sequence
$$
0\rightarrow \zz ^b \rightarrow \zz ^k \rightarrow NS(X,D)\rightarrow 0.
$$
Equation \eqref{nsutors} identifies the quotient of $NS(X,D)^{\rm sat}$ by $NS(X,D)$ with $NS(U)^{\rm tors}$.
\end{proof}

The weight equivalence relation induces an equivalence relation on sections: two sections are equivalent if and only if
their values are equivalent over each $\lambda \in \pp^1$. For this discussion, we work modulo the gauge group with $M_{\rm DH}(U)$. 

\begin{lemma}
\label{sectionsquotient}
There is a finite abelian group $K$ and an exact sequence
$$
0\rightarrow 
\Gamma (\pp ^1, T(1,\log )^b)^{\sigma}\rightarrow 
Gr ^W_2 \Gamma (\pp ^1, M_{\rm DH}(U))^{\sigma} \rightarrow K \rightarrow 0 .
$$
For any $p\in \aaa ^1$ the parabolic weight and residue give an exact sequence
$$
0\rightarrow 
\left( \frac{\rr\times\cc}{(1,-p)\zz} \right) ^b  \rightarrow 
Gr ^W_2 \Gamma (\pp ^1, M_{\rm DH}(U))^{\sigma} \rightarrow
K \rightarrow 0 .
$$
\end{lemma}
\begin{proof}
The second exact sequence comes from the first via Proposition \ref{prop-parabolicweight}. 
The map 
\begin{equation}
\label{interchange}
Gr ^W_2 \Gamma (\pp ^1, M_{\rm DH}(U))^{\sigma}\rightarrow  \Gamma (\pp ^1, Gr ^W_2M_{\rm DH}(U))^{\sigma}
\end{equation}
is injective. 
Suppose we have a $\sigma$-invariant section of $Gr ^W_2M_{\rm DH}(U)$. The obstruction to lifting it to a section
in $Gr ^W_2 \Gamma (\pp ^1, M_{\rm DH}(U))^{\sigma}$ lies in $H^1(\pp ^1,M_{\rm DH}(X)^o)$. In view of the exact sequence used
in Lemma \ref{uniglobal}, we have $H^1(\pp ^1,M_{\rm DH}(X)^o)=H^2(\pp ^1,A)$ which is discrete. Therefore, on the connected component
of the space of sections, the map \eqref{interchange} is surjective. There is a finite subgroup $K\subset NS(U)^{\rm tors}$
representing the components in the image of \eqref{interchange}. In fact $A=H^1(X,\zz )$ and there is an exact sequence of the form
$$
0\rightarrow 
\Gamma (\pp ^1, T(1,\log )^b)^{\sigma}\rightarrow 
Gr ^W_2 \Gamma (\pp ^1, M_{\rm DH}(U))^{\sigma} \rightarrow NS(U)^{\rm tors} \rightarrow H^2(\pp ^1,H^1(X,\zz ))
$$
I don't know whether there are any examples where the last connecting map is nonzero. 
\end{proof}

\subsection{Proof of Theorem \ref{pref-id}}
\label{sub-proof}

Injectivity is proven in Corollary \ref{pref-id-inj}. 

Suppose we are given a $\sigma$-invariant section in the target of the map $\Pp$. 
Lift it over the quotient of the action of the gauge group $\Gg$, to get a section 
$$
\epsilon \in \Gamma (\pp ^1,M_{\rm DH}(X,\log D))^{\sigma}.
$$
We would like to construct a harmonic bundle mapping to $\epsilon$. For this, we will use the
correspondence \cite{hbnc} \cite{Mochizuki} \cite{Mochizuki2} 
between harmonic bundles and parabolic logarthmic $\lambda$-connections for some fixed $\lambda \in  \aaa ^1$.
It would be sufficient to use the Higgs case $\lambda = 0$ or the de Rham case $\lambda = 1$ but it is interesting to
treat a general $\lambda$. 

The value $\epsilon (\lambda )$  corresponds to a logarithmic $\lambda$-connection $(E,\nabla )$, with residue ${\rm res}_{\lambda ,D_i}(\epsilon )$ along
each $D_i$. On the other hand, consider the parabolic weight parameter 
$$
b_i:=\varpi _{\lambda ,D_i}(\epsilon )\in  \rr.
$$
Put a parabolic structure onto $(E,\nabla )$ using these weights. This gives a parabolic logarithmic $\lambda$-connection $(E(\sum b_iD_i),\nabla )$.

We claim that $c_1( E(\sum b_iD_i))=0$. To see this, look at the exact sequence of Proposition \ref{weight-exact}. Take spaces of 
$\sigma$-invariant sections, and use the identification of Proposition \ref{prop-parabolicweight} at our fixed $\lambda$. In these terms, $\epsilon$ maps
to an element of 
$$
\ker \left( NS(X,D)^{\rm sat} \times (\rr \times \cc ) ^k \rightarrow H^{1,1}(X,\rr )\otimes _{\rr} (\rr \times \cc ) \right) .
$$
The coefficient in $NS(X,D)^{\rm sat}$ is $\zeta = c_1(E)$, whereas the coefficient in $\rr ^k$ is $(b_1,\ldots , b_k)$.
The image in the first factor $H^{1,1}(X,\rr )\otimes _{\rr} \rr$ is
$$
\varpi _{\lambda} (\Lambda _{\rm DH}(\zeta ) + \Sigma _{\rm DH}(b_1,\ldots , b_k) ).
$$
Notice that $\varpi  _{\lambda}\circ \Lambda _{\rm DH}$ is equal to the usual map 
$NS(X,D)^{\rm sat}\rightarrow H^{1,1}(X,\rr )$. This is because of the normalization condition that $\varpi _{\lambda}(1,0)=1$ used in 
\S \ref{sub-parabolicweight}. Similarly, 
$$
\varpi _{\lambda} \Sigma _{\rm DH}(b_1,\ldots , b_k)  = b_1[D_1]+ \ldots + b_k[D_k]\;\; \in H^{1,1}(X,\rr ).
$$
We conclude that
$$
c_1( E(\sum b_iD_i)) = c_1(E) + b_1[D_1]+\ldots + b_k[D_k] = 0
$$
as claimed. 

Then the harmonic theory for parabolic logarithmic connections \cite{EellsSampson} \cite{Corlette} \cite{DonaldsonApp} \cite{hbnc} 
\cite{Mochizuki2} \cite{Budur} provides a rank $1$ tame harmonic bundle $\Ee$ over $U$,
whose associated parabolic $\lambda$-connection is $(E(\sum b_iD_i),\nabla )$. By Theorem \ref{harmoniccompose}, the parabolic weight of $\Ee$ is the same as
the parabolic weight of the harmonic bundle, that is
$$
\varpi _{1,D_i}(\Pp (\Ee )) = b_i .
$$
This coincides with the parabolic weight of $\epsilon$. 
Furthermore, by construction the values of $\Pp (\Ee )$ and $\epsilon$ at $p=1$
are the same, both equal to $(E,\nabla )$. We conclude that $\Pp (\Ee ) = \varepsilon$, which concludes the proof of Theorem \ref{pref-id},
by the following lemma.

\begin{lemma}
\label{section-unique}
Suppose $p\in \pp ^1$, and suppose $\xi , \epsilon$ are two $\sigma$-invariant sections of $M_{\rm DH}(X,\log D)$. 
Suppose that for each $D_i$, the parabolic weights agree $\varpi _{p,D_i}(\xi ) = \varpi _{p,D_i}(\epsilon )$. Suppose furthermore
that $\xi (p) = \epsilon (p)$. Then $\xi = \epsilon$. 
\end{lemma}
\begin{proof}
The weight filtration exact sequence \eqref{exseq} gives an exact sequence on spaces of $\sigma$-invariant sections. 
Then, identify the space of sections of $T(1,\log )^k$ with $(\rr \times \cc )^k$ using $(\varpi _p, {\rm res}_p)$ as in Proposition \ref{prop-parabolicweight}. 
If $\xi (p)=\epsilon (p)$ then their residues at $p$ agree. By hypothesis the parabolic weight coordinates agree. Therefore, 
$\xi$ and $\epsilon$ go into the same section of $T(1,\log )^k$. 

They go to the same section of the discrete group $NS(X,D)^{\rm sat}$, because the values at $p$ are
the same by hypthesis. By Lemma \ref{weight2}, $\xi$ and $\epsilon$ go to the same section of $Gr^W_2 M_{\rm DH}(X,\log D)$. 
 
Therefore the difference $\xi \otimes \epsilon ^{-1}$ comes from a $\sigma$-invariant section of $M_{\rm DH}(X)$. 
As was noted in Lemma \ref{uniglobal},
the weight $1$ property of $M_{\rm DH}(X)$ says that the space of $\sigma$-invariant sections here maps isomorphically to any fiber. The condition 
$\xi (p) = \epsilon (p)$ thus implies that $\xi \otimes \epsilon ^{-1}$ is trivial. 
\end{proof}

\section{Strictness consequences}
\label{sec-strict}

One of the most useful things about weights in Hodge theory is that they lead to a notion of strictness. 
Here we formulate a conjecture which would be the corresponding strictness property coming from the weight two piece of 
the nonabelian $H^1$. Since it is just a conjecture, we consider representations of any rank. 

Suppose $(X,D)$ and $(Y, E)$ are smooth projective varieties with simple normal crossings divisors, such that $D$ has $k$ components and $E$ has $m$ components. 
Suppose $\Ff$ is some natural
construction from local systems on $U:=X-D$ to local systems on $V:=Y-E$. This could include any combination of pullbacks, higher direct images,
tensor products, duals, etc. For the present purposes, denote by $M_B(U)$ and $M_B(V)$ the full unions of spaces of representations of all ranks.
There will be a stratification of $M_B(U)$ into locally closed subsets such that $\Ff$ is algebraic on each stratum. Assume that this stratification is
maximal, that is $M_B(U)_{\alpha}$ is the full subset of representations $\rho$ on $U$ of a given rank, such that the image $\Ff (\rho )$ has a given rank on $V$. 

In the higher rank case, the eigenvalues of the local monodromy transformations may be considered all at once, with their multiplicities, as divisors on $\Gm ^{\perp}$. 
The group of such divisors is denoted $Div(\Gm ^{\perp})$, and for a divisor $D$ decomposing into $k$ irreducible components, the
full collection of residual data is a point in $Div(\Gm ^{\perp})^k$. 

\begin{conjecture}
\label{conj-strictness}
Let $M_B(U)_{\alpha}$ be a stratum on which $\Ff =\Ff _{\alpha}$ is defined as an algebraic map into $M_B(V)$. 
\newline
(1)\, There should be a diagram expressing the effect of the construction $\Ff _{\alpha}$ on residues: 
$$
\begin{array}{ccc}
M_B(U)_{\alpha} & \rightarrow & Div(\Gm ^{\perp})^k \\
\downarrow && \downarrow \\
M_B(V)_{\alpha} & \rightarrow & Div(\Gm ^{\perp})^m \, .
\end{array} 
$$
(2) \, The following strictness property holds: suppose $\rho _1, \rho _2 \in M_B(U)_{\alpha}$ are two semisimple representations such that
$\Ff _{\alpha} (\rho _1)$ and $\Ff _{\alpha}(\rho _2)$ have the same residues in $Div(\Gm ^{\perp})^m$. Then there exists a semisimple representation 
$\rho _3\in M_B(U)_{\alpha}$ such that $\rho _3$ has the same residues as $\rho _1$  in $Div(\Gm ^{\perp})^k$,
but $\Ff (\rho _3) \cong \Ff (\rho _2)$.
\end{conjecture}

To phrase it differently, this conjecture says that any variation of the image representation $\Ff (\rho )$, within a locus of representations on $V$ all
having the same residues, obtained by possibly varying the residues of $\rho$, can equally well be obtained while keeping the residues of $\rho$ fixed. 

It would be the analogue of the same statement in abelian Hodge theory for the diagram 
$$
\begin{array}{ccc}
H^1(U) & \rightarrow & Gr^W_2(H^1(U)) \\
\downarrow && \downarrow \\
H^1(V) & \rightarrow & Gr^W_2(H^1(V))
\end{array} .
$$
In the abelian case, pretty much the only possibility for the construction $\Ff$ is pullback for a map $V \rightarrow U$. 
The strictness statement says that if $a_1,a_2$ are classes in 
$H^1(U)$ whose pullbacks to $V$ have the same residues along $D$, then there is a class $a_3$ with the same residues as $a_2$,
whose pullback coincides with the pullback of $a_2$. 

Our observation of the weight two phenomenon in the case of rank one local systems should provide a proof of this conjecture for the rank one case. 
We don't discuss that here: it would go beyond the scope of the paper. 

One can also expect an infinitesimal formulation of the strictness property, which might be easier to prove. It would be the same statement, in the case where
$\rho _1$ and $\rho _2$ are infinitesimally close, and we would look for $\rho _3$ also infinitesimally close. This should be a consequence of having a 
mixed Hodge structure on the local deformation theory \cite{BrylinskiFoth} \cite{Foth} \cite{PridhamMHS} \cite{PridhamQl}, 
plus a compatibility of the construction $\Ff$ with this mixed Hodge structure. Again, this goes out of the
scope of the present discussion. 

One should also be able to formulate a similar conjecture for harmonic bundles with the parabolic residual data characterized by points in 
$Div(\frac{\rr}{\zz} \times \cc )^k$. 

We have been vague about what happens in the case of non-semisimple residues: is there a way to take into account the unipotent piece of the residue in the
strictness statement? It doesn't seem completely clear what is the right thing to say.

\bibliographystyle{amsalpha}

\begin{thebibliography}{A}


\bibitem{Arapura}
D. Arapura. 
{\em Geometry of cohomology support loci for local systems. I.}
J. Algebraic Geom. {\bf 6} (1997),  563-597.

\bibitem{BalajiBiswasNagaraj}
V. Balaji, I. Biswas, D.  Nagaraj.
{\em Principal bundles over projective manifolds with parabolic structure over a divisor.} 
Tohoku Math. J. {\bf 53} (2001), 337-367.

\bibitem{Biswas} Biswas, I. 
{\em Parabolic bundles as orbifold bundles}, Duke Math. J., {\bf 88} (1997), 305-325.

\bibitem{Biquard}
O. Biquard. {\em Sur les fibr\'es paraboliques sur une surface complexe.}
J. London Math. Soc. {\bf  53} (1996), 302-316.

\bibitem{Boalch}
P. Boalch. 
{\em From Klein to Painlev\'e via Fourier, Laplace and Jimbo. }
Proc. London Math. Soc. {\bf  90} (2005), 167-208.

\bibitem{BodenYokogawa}
H. Boden, K. Yokogawa. {\em Moduli spaces of parabolic Higgs bundles and parabolic $K(D)$ pairs over smooth curves, I.}
Internat. J. Math. {\bf 7} (1996), 573-598.

\bibitem{BrylinskiFoth}
J. L. Brylinski, P. Foth. 
{\em Moduli of flat bundles on open K\"ahler manifolds.} 
J. Algebraic Geom. {\bf 8} (1999), 147-168.  

\bibitem{Budur}
N. Budur. {\em Unitary local systems, multiplier ideals, and polynomial periodicity of Hodge numbers. }
Preprint \verb+arXiv:math/0610382+. 

\bibitem{CecottiVafa}
S. Cecotti, C. Vafa.
{\em Topological--anti-topological fusion.}
Nuclear Phys. B {\bf 367} (1991),  359-461. 

\bibitem{Connes}
A. Connes.  {\em Cohomologie cyclique et foncteur $Ext^n$.} Comptes Rendues Ac. Sci.
Paris S\'er. A-B, 296 (1983), 953-958.

\bibitem{Corlette}
K. Corlette.  {\em Flat $G$-bundles with canonical metrics.}  J. Diff. Geom. {\bf 28} (1988), 361-382.

\bibitem{DeligneLett}
P. Deligne. Letter to the author (March 20, 1989). 

\bibitem{Deninger}
C. Deninger. {\em On the $\Gamma$-factors attached to motives.} Invent. Math. {\bf 104} (1991), 245-263. 

\bibitem{Dimca}
A. Dimca. {\em Characteristic varieties and constructible sheaves.} \newline
Preprint \verb+arXiv:math/0702871v2+.

\bibitem{DimcaMaisonobeSaito}
A. Dimca, P. Maisonobe, M. Saito. 
{\em Spectrum and multiplier ideals of arbitrary subvarieties.}
Preprint \verb+arXiv:0705.4197v1+.

\bibitem{DimcaPapadimaSuciu}
A. Dimca, S. Papadima, A. Suciu. 
{\em Alexander polynomials: Essential variables and multiplicities.}
Preprint \verb+arXiv:0706.2499v1+.

\bibitem{DonaldsonApp}
S. Donaldson. 
{\em Twisted harmonic maps and the self-duality equations.}
Proc. London Math. Soc. {\bf 55} (1987), 127-131.

\bibitem{EellsSampson}
J. Eells, J. Sampson.{\em  Harmonic mappings of Riemannian manifolds.} Amer. J. Math. {\bf 86} (1964), 109-160.

\bibitem{Esnault}
H. Esnault. {\em Characteristic classes of flat bundles} Topology {\bf 27} (1988), 323–352.

\bibitem{EsnaultViehweg}
H. Esnault, E. Viehweg. 
{\em Deligne-Beilinson cohomology.} Beilinson's conjectures on special values of $L$-functions, 
Perspect. Math., 4, Academic Press, Boston (1988), 43-91. 

\bibitem{Foth}
P. Foth.
{\em Deformations of representations of fundamental groups of open Kähler manifolds.} 
J. Reine Angew. Math. {\bf 513} (1999), 17-32. 

\bibitem{Gajer}
P. Gajer. 
{\em Geometry of Deligne cohomology.}
Invent. Math. {\bf 127} (1997), 155-207. 

\bibitem{GoldmanXia}
W. Goldman, E. Xia. {\em Rank One Higgs Bundles and Representations of Fundamental Groups of Riemann Surfaces.}
Preprint {\tt math/0402429}.

\bibitem{GukovWitten}
S. Gukov, E. Witten. {\em Gauge Theory, Ramification, And The Geometric Langlands Program.} Preprint \verb+arXiv:hep-th/0612073v1+. 

\bibitem{Hertling}
C. Hertling. 
{\em $tt^{\ast}$ geometry, Frobenius manifolds, their connections, and the construction for singularities.} 
J. Reine Angew. Math. 555 (2003), 77-161. 

\bibitem{Hitchin}
N. Hitchin. 
{\em The self-duality equations on a Riemann surface.}
Proc. London Math. Soc. (3) {\bf 55} (1987),  59-126.

\bibitem{HitchinH}
N. Hitchin, A. Karlhede, U. Lindstr\"om, M. Ro\u{c}eck. 
{\em Hyperk\"{a}hler metrics and supersymmetry.}
Comm. Math. Phys. {\bf 108} (1987), 535-559. 
 
\bibitem{InabaIwasakiSaito}
M. Inaba, K. Iwasaki, Masa-Hiko Saito.
{\em Moduli of stable parabolic connections, Riemann-Hilbert correspondence and geometry of Painlev\'e 
equation of type VI, Part I.}
Publ. Res. Inst. Math. Sci. {\bf 42} (2006), 987-1089.

\bibitem{InabaIwasakiSaito2}
M. Inaba, K. Iwasaki, Masa-Hiko Saito.
{\em Moduli of stable parabolic connections, Riemann-Hilbert correspondence and geometry of Painlev\'e 
equation of type VI, Part II.}
Moduli spaces and arithmetic geometry. Adv. Stud. Pure Math. {\bf 45}, Math. Soc. Japan (2006), 387--432.

\bibitem{IyerSimpson}
J. Iyer, C. Simpson. {\em  The Chern character of a parabolic bundle, and a parabolic corollary of Reznikov's theorem.}
To appear, Progr. in Math. {\bf 265}, Birk\"auser (2007), 437-483.

\bibitem{JostZuo1}
J. Jost, K. Zuo.{\em  Harmonic maps and $SL(r,\cc )$ representations of fundamental groups of 
quasi-projective manifolds.} 
J. Algebraic Geom. {\bf 5} (1996), 77-106.

\bibitem{Kaledin}
D. Kaledin. {\em Non-commutative Hodge-to-de Rham degeneration via the method of Deligne-Illusie.} Preprint \verb+arXiv:math/0611623v3+.

\bibitem{Konno}
H. Konno. {\em Construction of the moduli space of stable parabolic Higgs bundles on a Riemann surface.} J. Math.
Soc. Japan {\bf 45} (1993), 253-276.

\bibitem{Li}
Jiayu Li. {\em  Hermitian-Einstein metrics and Chern number inequalities on parabolic stable bundles over K\"ahler 
manifolds.}  Comm. Anal. Geom. {\bf 8} (2000), 445-475.

\bibitem{LiNarasimhan}
J. Li, M.S. Narasimhan.  {\em Hermitian-Einstein metrics on parabolic stable bundles.}
Acta Math. Sin. {\bf 15} (1999), 93-114.

\bibitem{Libgober}
A. Libgober.
{\em Problems in topology of the complements to plane singular curves.}
Singularities in geometry and topology, World Sci. Publ. (2007),  370-387. 

\bibitem{LubotskyMagid}
A. Lubotsky, A. Magid. 
{\em Varieties of representations of finitely generated groups.}
Mem. Amer. Math. Soc. {\bf 58} (1985).

\bibitem{Mochizuki}
T. Mochizuki.
{\em Asymptotic behaviour of tame harmonic bundles and an application to pure twistor $D$-modules, parts I and II.}
Memoirs of the AMS {\bf 869}-{\bf 870} (2007).

\bibitem{Mochizuki2}
T. Mochizuki. 
{\em Kobayashi-Hitchin correspondence for tame harmonic bundles and an application.}
Ast\'erisque {\bf 309} (2006). 

\bibitem{Nakajima}
H. Nakajima. {\em Hyper-K\"ahler structures on moduli spaces of parabolic Higgs bundles on Riemann surfaces.} 
Moduli of vector bundles (Sanda, 
Kyoto, 1994), Lect. Notes Pure Appl. Math. {\bf 179} (1996), 199-208.

\bibitem{NeisendorferTaylor}
J. Neisendorfer, L. Taylor. {\em Dolbeault homotopy theory.} Trans. Amer. Math. Soc. 245 (1978), 183-210. 

\bibitem{Nitsure}
N. Nitsure. {\em Moduli of semistable logarithmic connections.} J. Amer. Math. Soc. {\bf 6} (1993), 597-609. 

\bibitem{Panov}
D. Panov. {\em Polyhedral K\"ahler manifolds.} Thesis, Ecole Polytechnique (2005). 

\bibitem{PridhamMHS}
J. Pridham. {\em The deformation theory of representations of the fundamental group of a smooth variety.}
Preprint \verb+arXiv:math/0401344v1+.

\bibitem{PridhamQl}
J. Pridham. 
{\em Deforming $l$-adic representations of the fundamental group of a smooth variety.}
J. Algebraic Geom. {\bf 15} (2006), 415-442. 

\bibitem{Sabbah}
C. Sabbah.  {\em Polarizable twistor ${\mathcal D}$-modules.} 
Ast\'erisque {\bf 300}, (2005).

\bibitem{Schafer}
L. Sch\"afer. {\em $tt*$-bundles in para-complex geometry, special para-Kähler manifolds and para-pluriharmonic maps}
Differential Geometry and its Applications {\bf 24} (2006),  60-89.

\bibitem{hbnc}
C. Simpson.{\em  Harmonic bundles on noncompact curves.} J. Amer. Math. Soc. {\bf 3}  (1990),  713-770.

\bibitem{icm}
C. Simpson. 
{\em Nonabelian Hodge theory.} 
Proceedings of the International Congress of Mathematicians, Vol. I, II (Kyoto, 1990),  Math. Soc. Japan, Tokyo, (1991), 747-756. 

\bibitem{hbls}
C. Simpson.{\em  Higgs bundles and local systems.} 
Publ. Math. I.H.E.S. {\bf 75} (1992), 5-95.

\bibitem{hfnac}
C. Simpson. 
{\em The Hodge filtration on nonabelian cohomology.} Algebraic geometry---Santa Cruz 1995, ,
Proc. Sympos. Pure Math., {\bf 62}, Part 2, A.M.S. (1997), 217-281.

\bibitem{twistor}
C. Simpson.{\em  Mixed twistor structures.} Preprint \verb+arXiv:alg-geom/9705006v1+. 

\bibitem{SteerWren}
B. Steer, A. Wren.{\em  The Donaldson-Hitchin-Kobayashi correspondence for 
parabolic bundles over orbifold surfaces.}
Canad. J. Math. {\bf 53} (2001), 1309-1339.

\bibitem{Thaddeus}
M. Thaddeus.{\em  Variation of moduli of parabolic Higgs bundles.} J. Reine Angew. Math. {\bf 547} (2002), 1-14. 

\bibitem{Voros}
A. Voros. 
{\em Probl\`eme spectral de Sturm-Liouville: le cas de l'oscillateur quartique.}
S\'eminaire Bourbaki Vol. 1982/83, 
Ast\'erisque, {\bf 105-106}, S.M.F., Paris, (1983), 95--104,. 

\bibitem{Yokogawa}
K. Yokogawa. {\em Compactification of moduli of parabolic sheaves and moduli of parabolic Higgs sheaves.}
J. Math. Kyoto Univ. {\bf 33} (1993), 451-504.


\end{thebibliography}

\end{document}